\documentclass[hidelinks,onefignum,onetabnum]{siamart220329}

\usepackage{lipsum}
\usepackage{amsfonts}
\usepackage{graphicx}
\usepackage{epstopdf}
\usepackage{algorithmic}
\ifpdf
  \DeclareGraphicsExtensions{.eps,.pdf,.png,.jpg}
\else
  \DeclareGraphicsExtensions{.eps}
\fi

\newsiamremark{remark}{Remark}
\newsiamremark{Ex}{Example}
\newsiamremark{hypothesis}{Hypothesis}
\crefname{hypothesis}{Hypothesis}{Hypotheses}
\newsiamthm{claim}{Claim}
\newsiamthm{conj}{Conjecture}

\author{Yuchang Liu\thanks{School of Mathematical Sciences,
         University of Science and Technology of China,
         Hefei, Anhui 230026, P.R. China.
  (\email{lissandra@mail.ustc.edu.cn}).}
\and Yan Jiang\thanks{ School of Mathematical Sciences,
         University of Science and Technology of China, Hefei, 
         Anhui 230026, P.R. China.  
         (\email{jiangy@ustc.edu.cn}). } 
\and Mengping Zhang\thanks{ School of Mathematical Sciences, University of Science and Technology of China, Hefei,
         Anhui 230026, P.R. China.  
         (\email{mpzhang@ustc.edu.cn}).} 
}

\usepackage{amsopn}

\makeatletter
\@ifundefined{Ex}{%
  \usepackage{amsfonts,amssymb,graphicx,epstopdf}%
  \@ifundefined{remark}{\newsiamremark{remark}{Remark}}{}%
  \newsiamremark{Ex}{Example}%
  \crefname{Ex}{Example}{Examples}%
  \author{Yuchang Liu \and Yan Jiang \and Mengping Zhang}%
}{}
\makeatother
\crefname{Ex}{Example}{Examples}
\Crefname{Ex}{Example}{Examples}
\crefname{operatorproperty}{property}{properties}
\crefformat{operatorproperty}{#2\textup{(#1)}#3}
\crefrangeformat{operatorproperty}{#3\textup{(#1)}#4--#5\textup{(#2)}#6}
\newsiamthm{assumption}{Assumption}
\crefname{assumption}{Assumption}{Assumptions}
\usepackage{booktabs}
\usepackage{mathtools}
\usepackage{mathrsfs}
\usepackage{tikz}
\usepackage{pgfplots}
\pgfplotsset{compat=1.18}
\usepgfplotslibrary{groupplots}
\usepackage{enumitem}
\title{High-Order Fully Discrete Entropy-Stable DG Methods with Compact Explicit Single-Step Time Discretizations}
\headers{FULLY DISCRETE ENTROPY-STABLE SINGLE-STEP DG}{Y. LIU, Y. JIANG, AND M. ZHANG}
\newcommand{\UU}{\mathcal U}
\newcommand{\FF}{\mathcal F}
\newcommand{\pre}{\mathrm{pre}}
\newcommand{\tar}{\mathrm{tar}}
\newcommand{\loc}{\mathrm{loc}}

\allowdisplaybreaks
\ifpdf\hypersetup{pdftitle={High-Order Fully Discrete Entropy-Stable DG Methods with Compact Explicit Single-Step Time Discretizations},pdfauthor={Yuchang Liu, Yan Jiang, Mengping Zhang}}\fi

\begin{document}
	\maketitle
	\begin{abstract}
		We develop a unified framework for high-order, fully discrete entropy-stable discontinuous Galerkin (DG) methods with compact explicit single-step time discretizations. At the end of each time step, the preliminary DG polynomial is scaled toward its cell average to enforce a strong cell entropy inequality while preserving conservation. This correction is feasible if and only if the updated cell average satisfies a weak cell entropy inequality, which standard third-order strong-stability-preserving Runge--Kutta DG updates can violate at arbitrarily small positive CFL numbers. For compact single-step schemes, including compact Runge--Kutta, Lax--Wendroff, and ADER-DG methods, we derive an explicit interval of feasible numerical entropy fluxes at each interface and prove that it is nonempty under a CFL condition. Projecting a compatible high-order target flux onto this interval guarantees the weak cell entropy inequality and preserves single-step high-order accuracy on smooth solutions, including near nondegenerate extrema. The limiter also enforces multiple entropy inequalities simultaneously. We also establish a Lax--Wendroff-type theorem. Numerical experiments for scalar conservation laws and the compressible Euler equations demonstrate high-order accuracy, multi-entropy control, and robustness for strong shocks.
	\end{abstract}

	\tableofcontents
	
	\section{Introduction}
	Hyperbolic conservation laws arise in compressible fluid dynamics,
	multiphase flows, and many related applications. Because their solutions can develop discontinuities in finite time even from smooth initial data, weak solutions are generally nonunique and must be supplemented by entropy admissibility conditions. Discontinuous Galerkin (DG) methods combine local polynomial spaces with conservative numerical fluxes to achieve high-order accuracy in smooth regions while capturing these discontinuities. In particular, the strong-stability-preserving (SSP) Runge--Kutta (RK) DG methods of Cockburn and Shu~\cite{cockburn1989tvb,cockburn1991runge} are widely used for hyperbolic problems.
	
	A central goal in the design of high-order schemes is to enforce discrete entropy inequalities without degrading accuracy on smooth solutions. Tadmor's framework~\cite{tadmor1987numerical,tadmor2003entropy} connects numerical dissipation to entropy production and forms the basis of entropy-conservative and entropy-stable schemes. For semidiscrete scalar DG methods with exact integration, Jiang and Shu~\cite{jiang1994cell} proved a cell entropy inequality for the square entropy. General convex entropies can be treated by combining summation-by-parts
    quadrature rules with entropy-conservative flux
    differencing~\cite{chen2017entropy,chan2018discretely}.
    An alternative approach, proposed by Abgrall~\cite{abgrall2018general},
    introduces local corrections to the spatial residuals to enforce a
    discrete entropy equality or inequality while preserving conservation,
    without requiring special quadrature rules.
    See Chen and Shu~\cite{chen2020review} for a review. Semidiscrete entropy stability, however, does not automatically carry over to explicit high-order time integration.
	
	Fully discrete entropy stability can be achieved via implicit time stepping, such as backward Euler~\cite{tadmor2003entropy}, entropy-variable DG formulations with compatible time integrators~\cite{colombo2023}, or space--time DG discretizations~\cite{hiltebrand2014entropy}. To avoid the coupled nonlinear solves required by implicit methods, relaxation Runge--Kutta schemes~\cite{ranocha2020relaxation} adjust the step completion via a scalar relaxation parameter to preserve global entropy behavior. Enforcing cellwise entropy inequalities or controlling multiple entropies within that framework requires local relaxation extensions with additional scalar root-finding in each cell~\cite{ranocha2020local}; see also Gaburro et al.~\cite{gaburro2023high} for ADER schemes combining spatial entropy corrections with temporal relaxation.
	
	In our previous work~\cite{liu2026limiter,liu2026mhd}, we enforced fully discrete cell entropy inequalities through an explicit, cell-average-preserving scaling of the DG polynomial. For a forward Euler step, a proper numerical entropy flux guarantees a weak cell entropy inequality for the updated cell averages, which provides a feasible upper bound for terminal polynomial scaling. This algebraic limiter requires no nonlinear solver, accommodates any finite family of convex entropy pairs simultaneously, and yields a Lax--Wendroff-type convergence theorem. Related cell-average-preserving scaling strategies have been developed in Wu's entropy--positivity--oscillation (EPO) framework~\cite{wu2026epounifiedframeworkentropy} and applied to relativistic hydrodynamics~\cite{xuwu2026} and compressible magnetohydrodynamics~\cite{caowuyuan2026}.
	
	To achieve high-order temporal accuracy, the limiter-based schemes in~\cite{liu2026limiter,liu2026mhd,caowuyuan2026} rely on SSP multistep time integration. Because an SSP multistep update is a convex combination of forward Euler steps from several preceding time levels, the resulting entropy inequality bounds the new cell entropy only by a convex combination of entropies from multiple past levels, rather than directly between two successive time levels. By contrast, applying the scaling limiter at every internal stage of a classical SSP-RK method enforces a two-level inequality by convexity, but stagewise limiting can degrade spatiotemporal accuracy to first order~\cite{liu2026limiter,wu2026epounifiedframeworkentropy}. We therefore seek explicit single-step DG methods that satisfy a two-level strong cell entropy inequality through a single cell-average-preserving correction at the end of each time step.
	
	By Jensen's inequality, the entropy of the preliminary cell average is the minimum cell entropy achievable under any cell-average-preserving correction. On a periodic domain, if the total entropy of the preliminary cell averages exceeds the total discrete entropy at the beginning of the current time step, no terminal cell-average-preserving limiter can prevent a global entropy increase. We show that standard SSP-RK3 DG methods with unlimited internal stages can violate this necessary condition at arbitrarily small positive CFL numbers. In our counterexamples, successive DG stages propagate interface slopes across multiple cells within a single step, producing an overshoot in a cell not adjacent to the initial discontinuity.
	
	This obstruction motivates the use of compact explicit single-step DG discretizations whose full-step update in each cell depends only on that cell and its immediate face neighbors. A prime example is the compact RKDG (cRKDG) method of Chen, Sun, and Xing~\cite{chen2023}, which replaces the inter-element DG operator in internal stages with a cell-local derivative projection and evaluates numerical fluxes only in the final stage. In a concurrent independent preprint, Sun~\cite{sun2026fully} studied entropy-stable cRKDG schemes with nonnegative stage weights using a direct stage-averaged Lax--Friedrichs entropy flux, leaving accuracy preservation open when the limiter is active.
    
    In this work, we develop a unified entropy-stabilization framework for general compact explicit single-step DG methods, encompassing cRKDG, Lax--Wendroff DG (LWDG)~\cite{qiu2005}, and ADER-DG schemes based on local space--time predictors~\cite{dumbser2008,busto2020}. At each cell interface, we use the actual full-step conservative flux together with the old cell averages and discrete cell entropies to define an interval of feasible numerical entropy fluxes. Under suitable assumptions, we prove that these feasible intervals are nonempty under a CFL restriction. Projecting a compatible high-order target entropy flux onto the feasible interval then guarantees the weak cell entropy inequality needed for terminal scaling.
	
	The main contributions of this paper are summarized as follows:
	\begin{itemize}[leftmargin=*]
		\item We construct a family of counterexamples for standard SSP-RK3 DG showing that, at any prescribed sufficiently small positive CFL number, the full-step cell averages can violate global entropy nonincrease, ruling out any terminal cell-average-preserving entropy correction.
		
		\item We derive explicit intervals of feasible numerical entropy fluxes from the full-step conservative flux for general compact single-step DG methods, including cRKDG, LWDG, and ADER-DG, and prove that they are nonempty for one-dimensional scalar conservation laws under a CFL condition. Combined with terminal scaling, this yields conservation- and compactness-preserving schemes satisfying two-level strong cell entropy inequalities for one or more prescribed convex entropies.
		
		\item By projecting a compatible high-order target flux onto the feasible interval and constructing a unified Lax--Friedrichs target, we prove $(k+1)$-th order single-step local accuracy preservation for $k\ge2$ both away from critical points and near nondegenerate smooth extrema.
	\end{itemize}
	
	The remainder of the paper is organized as follows. \Cref{sec:entropy} formulates the fully discrete cell entropy conditions and establishes the SSP-RK3 counterexample. \Cref{sec:framework} presents the compact single-step framework, proves the nonemptiness of the feasible intervals and local accuracy preservation, and details the cRKDG, LWDG, and ADER-DG realizations. \Cref{sec:experiments} reports numerical experiments for scalar equations and the Euler system, and \cref{sec:conclusions} concludes the paper.
	
	\section{Entropy analysis and an obstruction to terminal correction}\label{sec:entropy}
	Following~\cite{liu2026limiter}, we review the continuous and fully discrete entropy conditions, formulate the cell-average feasibility requirement for terminal scaling, and demonstrate why standard multi-stage RKDG updates can violate this requirement.
	
\subsection{Conservation laws and entropy inequalities}
Consider a hyperbolic system of conservation laws on a periodic domain $\Omega\subset\mathbb R^d$:
\begin{equation}\label{eq:systemprelim}
 \partial_t\mathbf u+\sum_{j=1}^d\partial_{x_j}\mathbf f_j(\mathbf u)=\mathbf0,
 \qquad\mathbf u(\mathbf x,0)=\mathbf u_0(\mathbf x),
\end{equation}
where $\mathbf u$ takes values in a convex admissible state set $\mathcal D\subset\mathbb R^m$.

\begin{definition}[Entropy pair]
A convex function $\UU\in C^2(\mathcal D)$, together with entropy fluxes $\FF_j:\mathcal D\to\mathbb R$, forms an entropy pair for \cref{eq:systemprelim} if
$\nabla\UU(\mathbf u)D\mathbf f_j(\mathbf u)=\nabla\FF_j(\mathbf u)$ for $j=1,\ldots,d$.
\end{definition}

Let $\mathbf v=\nabla\UU(\mathbf u)^{\mathrm T}$ denote the entropy variables and $\boldsymbol\FF=(\FF_1,\ldots,\FF_d)^{\mathrm T}$. Smooth solutions of \cref{eq:systemprelim} satisfy the additional conservation law $\partial_t\UU(\mathbf u)+\nabla_{\mathbf x}\cdot\boldsymbol\FF(\mathbf u)=0$, whereas admissible weak solutions satisfy $\partial_t\UU(\mathbf u)+\nabla_{\mathbf x}\cdot\boldsymbol\FF(\mathbf u)\le0$ in the sense of distributions; that is, for every nonnegative smooth periodic test function $\varphi$ with compact time support,
\begin{equation}\label{eq:weakentropy}
 \int_0^\infty\!\int_\Omega
 [\UU(\mathbf u)\varphi_t+\nabla_{\mathbf x}\varphi\cdot\boldsymbol\FF(\mathbf u)]
 \,\mathrm d\mathbf x\,\mathrm dt
 +\int_\Omega\UU(\mathbf u_0)\varphi(\mathbf x,0)\,\mathrm d\mathbf x\ge0.
\end{equation}
On a periodic domain $\Omega$, this entropy inequality implies that the total entropy is nonincreasing in time:
$\int_\Omega\UU(\mathbf u(\mathbf x,t_2))\,\mathrm d\mathbf x
\le\int_\Omega\UU(\mathbf u(\mathbf x,t_1))\,\mathrm d\mathbf x$ for $t_2\ge t_1$. Our goal is to preserve it in the discrete sense. 

Theoretical analysis focuses on the one-dimensional scalar conservation law
\begin{equation}\label{eq:scalarprelim}
 u_t+f(u)_x=0,\qquad u(x,0)=u_0(x),
\end{equation}
for which $\FF'(u)=\UU'(u)f'(u)$, with entropy variable $v=\UU'$ and entropy potential $\psi=vf-\FF$.

\subsection{DG approximations and fully discrete entropy stability}
	Partition the one-dimensional periodic domain $\Omega$ into uniform cells $I_i=[x_{i-1/2},x_{i+1/2}]$ of width $h$, and define the broken polynomial space
	\[
	V_h^k=\{w:w|_{I_i}\in \mathbb{P}^k(I_i)\text{ for every }i\}.
	\]
	The standard semidiscrete DG method seeks $u_h\in V_h^k$ such that, on each cell $I_i$,
	\begin{equation}\label{eq:dgprelim}
		\int_{I_i}\partial_tu_h\,w\,\mathrm{d}x
		=\int_{I_i}f(u_h)w_x\,\mathrm{d}x
		-\hat f_{i+1/2}w^-_{i+1/2}
		+\hat f_{i-1/2}w^+_{i-1/2},
		\qquad w\in \mathbb{P}^k(I_i),
	\end{equation}
	where $\hat f_{i+1/2}=\hat f\bigl(u_h^-(x_{i+1/2}),u_h^+(x_{i+1/2})\bigr)$ is a consistent monotone numerical flux~\cite{shu2009discontinuous}, and superscripts $-$ and $+$ denote left and right interface traces. We write $\{a\}=(a^-+a^+)/2$ and $[a]=a^+-a^-$ for the interface average and jump, and denote the spatial DG operator defined by \cref{eq:dgprelim} by $\mathcal L^{\mathrm{DG}}$. A canonical choice is the Lax--Friedrichs (LF) flux
	\begin{equation}\label{eq:LF}
		\hat f(a,b)=\dfrac{f(a)+f(b)}2-\dfrac\alpha2(b-a),
	\end{equation}
	where $\alpha>0$ bounds $|f'(u)|$.
	
	Let $u_h^n\in V_h^k$ denote the numerical solution at $t^n$, with exact cell average
	\begin{equation}\label{eq:cellmeanexact}
		\bar u_i^n=\dfrac1h\int_{I_i}u_h^n(x)\,\mathrm{d}x.
	\end{equation}
	On each cell $I_i$, we introduce an $N_{\mathrm{G}}$-point Gauss--Legendre quadrature rule with nodes $x_{i,\ell}$ and normalized weights $w_\ell>0$ with $\sum_{\ell=1}^{N_{\mathrm{G}}}w_\ell=1$. Throughout the paper we take $N_{\mathrm{G}}\ge k+1$, so that the quadrature is exact on $\mathbb P^{2k}(I_i)$, and define the discrete cell entropy by
	\begin{equation}\label{eq:oldentropynotation}
		\widetilde{\UU}_i^n
		=\sum_{\ell=1}^{N_{\mathrm{G}}}w_\ell\UU\bigl(u_h^n(x_{i,\ell})\bigr),\qquad
		\bar u_i^n=\sum_{\ell=1}^{N_{\mathrm{G}}}w_\ell u_h^n(x_{i,\ell}).
	\end{equation}
	Because the weights are positive and sum to one, Jensen's inequality gives $\UU(\bar u_i^n)\le\widetilde{\UU}_i^n$. Unlike the Gauss--Lobatto quadrature required in~\cite{liu2026limiter}, the quadrature nodes in \cref{eq:oldentropynotation} need not include the cell endpoints $x_{i\pm1/2}$.
	
	Let $\Delta t=t^{n+1}-t^n$ and $\lambda=\Delta t/h$. A single-step explicit method, such as forward Euler with $u_h^{n+1,(\pre)}=u_h^n+\Delta t\,\mathcal L^{\mathrm{DG}}(u_h^n)$, advances $u_h^n$ to a preliminary polynomial $u_h^{n+1,(\pre)}\in V_h^k$ whose cell average satisfies the conservative update
	\begin{equation}\label{eq:mean}
		\bar u_i^{n+1}=\bar u_i^n
		-\lambda(\hat f_{i+1/2}-\hat f_{i-1/2}),
	\end{equation}
	where $\hat f_{i+1/2}$ is the effective time-averaged conservative interface flux over the full step. We now review the weak-to-strong framework in \cite{liu2026limiter, wu2026epounifiedframeworkentropy}. 
	
	\begin{definition}[Strong cell entropy inequality]
		A DG update satisfies the fully discrete strong cell entropy inequality for $(\UU,\FF)$ if there exists a shared numerical entropy flux $\hat{\FF}$, consistent with $\FF$, such that
		\begin{equation}\label{eq:desiredcellentropy}
			\widetilde{\UU}_i^{n+1}\le
			\widetilde{\UU}_i^n
			-\lambda(\hat{\FF}_{i+1/2}-\hat{\FF}_{i-1/2})
		\end{equation}
		for every cell.
	\end{definition}
	
However, $u_h^{n+1,(\pre)}$ may violate \cref{eq:desiredcellentropy}. We then apply a terminal correction that preserves its cell average by scaling toward $\bar u_i^{n+1}=h^{-1}\int_{I_i}u_h^{n+1,(\pre)}\,\mathrm{d}x$~\cite{zhang2010positivity,zhang2012maximum,wu2018provably,liu2026limiter,wu2026epounifiedframeworkentropy}:
\begin{equation}\label{eq:scalingprelim}
 u_h^{n+1}|_{I_i}=\bar u_i^{n+1}+\theta_i\bigl(u_h^{n+1,(\pre)}-\bar u_i^{n+1}\bigr),\qquad 0\le\theta_i\le1.
\end{equation}
Define
\begin{equation}\label{eq:threeentropiesprelim}
\begin{gathered}
 \UU_i^{\mathrm{1st}}=\UU(\bar u_i^{n+1}),\qquad
 \UU_i^{\mathrm{high}}=\sum_{\ell=1}^{N_{\mathrm{G}}}w_\ell\UU\bigl(u_h^{n+1,(\pre)}(x_{i,\ell})\bigr),\\
 \UU_i^{\mathrm{up}}=\widetilde{\UU}_i^n-\lambda(\hat{\FF}_{i+1/2}-\hat{\FF}_{i-1/2}).
\end{gathered}
\end{equation}
Jensen's inequality and convexity give
\[
 \UU_i^{\mathrm{1st}}\le\widetilde{\UU}_i^{n+1}
 \le(1-\theta_i)\UU_i^{\mathrm{1st}}+\theta_i\UU_i^{\mathrm{high}}.
\]
The lower bound holds for any cell-average-preserving correction and is attained by \cref{eq:scalingprelim} at $\theta_i=0$. Thus a scaling factor enforcing \cref{eq:desiredcellentropy} exists if and only if the updated cell average satisfies the \textit{weak cell entropy inequality}
\begin{equation}\label{eq:localfeasibilityprelim}
 \UU_i^{\mathrm{1st}}\le\UU_i^{\mathrm{up}}.
\end{equation}
Whenever this condition holds, setting $\theta_i=1$ if $\UU_i^{\mathrm{high}}\le\UU_i^{\mathrm{up}}$ and $\theta_i=(\UU_i^{\mathrm{up}}-\UU_i^{\mathrm{1st}})/(\UU_i^{\mathrm{high}}-\UU_i^{\mathrm{1st}})$ otherwise yields \cref{eq:desiredcellentropy}. This framework reduces the enforcement of the strong cell entropy inequality to finding a shared numerical entropy flux satisfying \cref{eq:localfeasibilityprelim}.
	Summing \cref{eq:desiredcellentropy} over a periodic mesh cancels the shared interface fluxes and yields global entropy stability:
	\begin{equation}\label{eq:globalES}
		\sum_i h\widetilde{\UU}_i^{n+1}\le\sum_i h\widetilde{\UU}_i^n.
	\end{equation}
	Similarly, summing the weak cell entropy inequality \cref{eq:localfeasibilityprelim} yields a necessary global condition for the existence of any feasible shared entropy flux.
	
	\begin{corollary}[Necessary global condition]\label{prop:necessary}
		On a periodic mesh, if a shared numerical entropy flux satisfies $\UU_i^{\mathrm{1st}}\le\UU_i^{\mathrm{up}}$ in every cell, then
		\begin{equation}\label{eq:necessaryglobal}
			\sum_i h\UU(\bar u_i^{n+1})\le\sum_i h\widetilde{\UU}_i^n.
		\end{equation}
	\end{corollary}
	\begin{proof}
		Summing \cref{eq:localfeasibilityprelim} over the periodic mesh gives
		\[
		\sum_i h\UU_i^{\mathrm{1st}}
		\le\sum_i h\UU_i^{\mathrm{up}}
		=\sum_i h\widetilde{\UU}_i^n,
		\]
		since the shared numerical entropy fluxes cancel.
	\end{proof}
    
	Indeed, since $\bar u_i^{n+1}$ is fixed by the conservative preliminary update \cref{eq:mean}, no choice of numerical entropy fluxes and scaling factors $\theta_i\in[0,1]$ can achieve \cref{eq:desiredcellentropy} if \cref{eq:necessaryglobal} is violated.
	
	\subsection{A counterexample for the standard SSP-RK3 DG method} 
	Under a suitable CFL condition, a single forward Euler DG step satisfies the weak cell entropy inequality \cref{eq:localfeasibilityprelim} with a proper numerical entropy flux~\cite{liu2026limiter}. Although standard SSP-RK methods are convex combinations of forward Euler steps, applying the scaling limiter after each internal stage can reduce the accuracy to first order, whereas leaving the internal stages unlimited can cause the full-step update to violate the necessary global condition \cref{eq:necessaryglobal}.
	
	To see this, apply the classical three-stage SSP-RK3 scheme to \cref{eq:dgprelim} without internal stage limiting:
	\begin{align*}
		u_h^{(1)}&=u_h^n+\Delta t\mathcal L^{\mathrm{DG}}(u_h^n),\\
		u_h^{(2)}&=\dfrac34u_h^n+\dfrac14\bigl(u_h^{(1)}+
		\Delta t\mathcal L^{\mathrm{DG}}(u_h^{(1)})\bigr),\\
		u_h^{n+1,(\pre)}&=\dfrac13u_h^n+\dfrac23\bigl(u_h^{(2)}+
		\Delta t\mathcal L^{\mathrm{DG}}(u_h^{(2)})\bigr).
	\end{align*}
	Consider Burgers' equation with $f(u)=u^2/2$ and the exponential entropy pair $\UU(u)=e^u$, $\FF(u)=(u-1)e^u$, discretized by $\mathbb P^1$ DG with exact volume integration on a uniform periodic mesh $I_i=[(i-1)h,ih]$ of $[0,1]$ with $h=1/N$ and even $N\ge12$. For any prescribed CFL number $0<\nu\le3/8$, set
	\[
	s=\dfrac{2\nu}{3}\in\left(0,\dfrac14\right],\qquad
	A=\dfrac{8}{s^2}=\dfrac{18}{\nu^2}\ge128,\qquad
	\alpha=\dfrac{3A}{2},\qquad
	\Delta t=\dfrac{\nu h}{\alpha}=\dfrac{sh}{A},
	\]
	and take the piecewise constant initial data
	\[
	u_h^0(x)=\begin{cases}
		0,&0\le x<1/2,\\
		A,&1/2\le x<1,
	\end{cases}
	\]
	using the LF flux \cref{eq:LF} with the same parameter $\alpha$ at all interfaces and stages.
	
	Writing any stage polynomial on $I_i$ in the local coordinate $\xi=2(x-x_i)/h\in[-1,1]$ as $w|_{I_i}=m_i+b_i\xi$, testing \cref{eq:dgprelim} with $1$ and $\xi$ yields the exact $\mathbb P^1$ modal evolution equations
	\[
	h\frac{\mathrm{d}m_i}{\mathrm{d}t}=-(\hat f_{i+1/2}-\hat f_{i-1/2}),\qquad
	h\frac{\mathrm{d}b_i}{\mathrm{d}t}=3m_i^2+b_i^2-3(\hat f_{i+1/2}+\hat f_{i-1/2}),
	\]
	where the interface fluxes are evaluated at the traces $m_i\pm b_i$. Across the three stages, the slopes generated in the cells adjacent to the upward jump at $x=1/2$ propagate into the second cell to the right of the jump, $I_{N/2+2}$. Explicitly evaluating the three stages yields
	\begin{equation}\label{eq:overshoot}
		\bar u_{N/2+2}^{1}=A\bigl(1+\delta(s)\bigr)>A+2,
	\end{equation}
	where, for $0<s\le1/4$,
	\[
	\begin{aligned}
		\delta(s)&=\frac54s^2-\frac{41}{12}s^3-\frac{23}{12}s^4+\frac{125}{24}s^5+\frac32s^6+\frac16s^7\\
		&>\left(\frac54-\frac{41}{48}-\frac{23}{192}\right)s^2=\frac{53}{192}s^2>\frac{s^2}{4}.
	\end{aligned}
	\]
	Furthermore, all interface traces entering the three stage evaluations remain in $[-A/2,3A/2]$, so $\alpha=3A/2$ is a valid wave-speed bound throughout the step.
	
	Because each $\mathcal L^{\mathrm{DG}}$ evaluation couples a cell only to its immediate neighbors, three stages can modify at most the twelve cells within three cells of the two initial jumps:
	\[
	\mathcal S
	=\{1,2,3,N-2,N-1,N\}
	\cup
	\left\{\frac{N}{2}-2,\ldots,\frac{N}{2}+3\right\},
	\]
	consisting of six cells with initial value $0$ and six with initial value $A$, as shown in \cref{fig:rk3-affected-cells}.
	
	\begin{figure}[!htbp]
		\centering
		\begin{tikzpicture}[x=0.65cm,y=0.70cm,font=\small,line width=0.4pt]
			\foreach \x in {0,1,2,4.5,5.5,6.5,7.5,8.5,9.5,12,13,14} {
				\fill[gray!15] (\x,0) rectangle ++(1,0.22);
			}
			\foreach \a/\b in {0/3,4.5/10.5,12/15} {
				\draw (\a,0)--(\b,0);
			}
			\foreach \x in {0,1,2,3,4.5,5.5,6.5,7.5,8.5,9.5,10.5,12,13,14,15} {
				\draw (\x,0)--(\x,0.38);
			}
			\node at (3.75,0.18) {$\cdots$};
			\node at (11.25,0.18) {$\cdots$};
			\draw[fill=gray!45,line width=0.9pt] (8.5,0) rectangle (9.5,0.38);
			\draw[->,line width=0.6pt] (9,1.25)--(9,0.44);
			\node[anchor=south] at (9,1.27) {$I_{N/2+2}$};
			\draw[line width=1pt] (0,1)--(3,1);
			\draw[line width=1pt] (4.5,1)--(7.5,1)--(7.5,2.6)--(10.5,2.6);
			\draw[line width=1pt] (12,2.6)--(15,2.6);
			\node at (3.75,1) {$\cdots$};
			\node at (11.25,2.6) {$\cdots$};
			\node[anchor=east] at (-0.15,1) {$0$};
			\node[anchor=west] at (15.15,2.6) {$A$};
			\draw[densely dashed,gray] (7.5,0.38)--(7.5,1);
			\node[anchor=north] at (0,-0.12) {$x=0$};
			\node[anchor=north] at (7.5,-0.12) {$x=1/2$};
			\node[anchor=north] at (15,-0.12) {$x=1$};
			\foreach \a/\b in {0/3,4.5/10.5,12/15} {
				\draw (\a,-0.68)--(\a,-0.88)--(\b,-0.88)--(\b,-0.68);
			}
			\node[anchor=north] at (7.5,-1.02) {$\mathcal S$};
		\end{tikzpicture}
		\caption{Initial data and the twelve potentially affected cells,
			$\mathcal S$. The dark cell is $I_{N/2+2}$, whose updated cell average
			exceeds $A+2$. The endpoints are identified periodically; dots indicate
			omitted constant regions.}
		\label{fig:rk3-affected-cells}
	\end{figure}
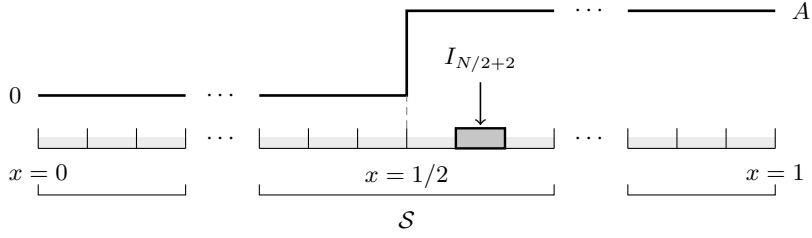
	
	Since $A\ge128$ and $e^2>7$, the single cell $I_{N/2+2}\in\mathcal S$ has more entropy after the step than all twelve cells in $\mathcal S$ combined initially:
	\[
	\sum_{i\in\mathcal S}h\UU(\bar u_i^{1})
	>he^{A+2}>7he^A>h(6+6e^A)=\sum_{i\in\mathcal S}h\widetilde{\UU}_i^0.
	\]
	Since all cells outside $\mathcal S$ are unchanged, $\sum_i h\UU(\bar u_i^{1})>\sum_i h\widetilde{\UU}_i^0$, violating \cref{prop:necessary}.
	
	\subsection{Compact RKDG and the preceding counterexample}\label{sec:crkstructure}
	The overshoot in $I_{N/2+2}$ above is driven by the expansion of the spatial stencil across successive RK stages. In the compact RKDG (cRKDG) method of Chen, Sun, and Xing~\cite{chen2023}, the inter-cell DG operator in internal stages is replaced by the cell-local operator
	\[
	\mathcal L^{\loc}w=-\Pi_k\partial_xf(w),
	\]
	where $\Pi_k$ is the cellwise $L^2$ projection onto $\mathbb{P}^k$ and $\partial_x$ is taken within each cell. With explicit RK coefficients $a_{r\ell}$ and $b_r$ satisfying $\sum_{r=1}^R b_r=1$, the stages and preliminary update read
	\begin{align*}
		u_h^{(r)}&=u_h^n+\Delta t\sum_{\ell<r}a_{r\ell}
		\mathcal L^{\loc}u_h^{(\ell)},\qquad r=1,\ldots,R,\\
		u_h^{n+1,(\pre)}&=u_h^n+\Delta t\sum_{r=1}^R b_r
		\mathcal L^{\mathrm{DG}}u_h^{(r)}.
	\end{align*}
	Here each internal stage $u_h^{(r)}|_{I_i}$ depends solely on $u_h^n|_{I_i}$, while neighboring cells interact only through the final conservative update. For third-order cRKDG, we use Heun's coefficients $a_{21}=1/3$, $a_{32}=2/3$, and $(b_1,b_2,b_3)=(1/4,0,3/4)$.
	
	For the piecewise constant initial data in the counterexample, $\mathcal L^{\loc}u_h^n=0$ everywhere, so $u_h^{(r)}=u_h^n$ for all $r$ and the preliminary cRKDG update reduces to a single forward Euler step:
	\[
	u_h^{n+1,(\pre)}=u_h^n+\Delta t\,\mathcal L^{\mathrm{DG}}u_h^n.
	\]
	Only the four cells immediately adjacent to the two jumps change their averages:
	\[
	\begin{aligned}
		\bar u_1^{n+1}&=\dfrac{2\nu A}{3},
		&\bar u_{N/2+1}^{n+1}&=\left(1-\dfrac{2\nu}{3}\right)A,\\
		\bar u_{N/2}^{n+1}&=\dfrac{\nu A}{3},
		&\bar u_N^{n+1}&=\left(1-\dfrac{\nu}{3}\right)A.
	\end{aligned}
	\]
	Since $0<\nu\le3/8$, each pair has the form $\theta A$ and $(1-\theta)A$ with $\theta\in[0,1]$, so convexity gives $\UU(\theta A)+\UU((1-\theta)A)\le\UU(0)+\UU(A)$ and the necessary global condition \cref{eq:necessaryglobal} holds. 
	
	Motivated by this property, Sun~\cite{sun2026fully} recently studied an entropy-stable cRKDG scheme using a direct stage-averaged entropy flux, which is restricted to cRKDG with nonnegative weights and lacks an accuracy guarantee, as shown in \cref{rem:rk-averaged-flux}. Here we develop a unified framework for general compact single-step schemes with provable accuracy preservation.
	
	\section{A unified framework for compact explicit single-step DG methods}\label{sec:framework}
	We now establish cellwise entropy stability and local accuracy preservation for a general class of compact explicit single-step DG methods. Throughout this section, $C>0$ denotes a generic constant independent of $h$ and $\Delta t$ and uniform over the stated CFL range.
	
	\subsection{Compact explicit single-step DG methods}\label{sec:compactstructure}
	
	First, we give a precise definition of the compact explicit single-step DG schemes studied in this work.
	
	\begin{definition}\label{def:compact}
		A fully discrete DG method is called a compact explicit single-step method if, for every $w\in V_h^k$, its preliminary solution $u_h^{n+1,(\pre)}$ satisfies
		\begin{equation}\label{eq:compactstructure}
			\begin{aligned}
				\int_{I_i}u_h^{n+1,(\pre)}w\,\mathrm{d}x
				={}&\int_{I_i}u_h^n w\,\mathrm{d}x
				+\Delta t\int_{I_i}f_h^\star\,\partial_xw\,\mathrm{d}x\\
				&-\Delta t\left(
				\hat f_{i+1/2}w^-_{i+1/2}
				-\hat f_{i-1/2}w^+_{i-1/2}\right),
			\end{aligned}
		\end{equation}
		where the effective volume flux $f_h^\star$ and interface flux $\hat f_{i+1/2}$ satisfy the following requirements:
		\begin{enumerate}[leftmargin=*]
			\item {Cell-local volume flux:} $f_h^\star=f_h^\star(u_h^n|_{I_i})$ depends only on the old polynomial in $I_i$, and $f_h^\star=f(c)$ whenever $u_h^n|_{I_i}\equiv c$.
			
			\item {Two-cell interface flux:} $\hat f_{i+1/2}=\hat f_{i+1/2}(u_h^n|_{I_i},u_h^n|_{I_{i+1}})$ depends only on the old polynomials in the two adjacent cells, and $\hat f_{i+1/2}=f(c)$ whenever $u_h^n|_{I_i\cup I_{i+1}}\equiv c$.
			
			\item {High-order consistency:} For any smooth exact solution $u(x,t)$, if $\|u_h^n-u(\cdot,t_n)\|_{L^\infty(I_j)}\le Ch^{k+1}$ on the relevant cells, then
			\[
			\begin{aligned}
				\left\|f_h^\star-\frac1{\Delta t}
				\int_{t_n}^{t_{n+1}}f\bigl(u(\cdot,t)\bigr)\,\mathrm{d}t
				\right\|_{L^\infty(I_i)}
				&\le Ch^{k+1},\\
				\left|\hat f_{i+1/2}-\frac1{\Delta t}
				\int_{t_n}^{t_{n+1}}f\bigl(u(x_{i+1/2},t)\bigr)\,\mathrm{d}t
				\right|
				&\le Ch^{k+1}.
			\end{aligned}
			\]
		\end{enumerate}
	\end{definition}
	
	\begin{remark}
		The cRKDG scheme of \cref{sec:crkstructure} fits \cref{def:compact} with
		\begin{equation}\label{eq:cRKflux}
			f_h^\star=\sum_{r=1}^R b_r f(u_h^{(r)}),\qquad
			\hat f_{i+1/2}=\sum_{r=1}^R b_r
			\hat f\bigl(u_h^{(r)}(x_{i+1/2}^-),u_h^{(r)}(x_{i+1/2}^+)\bigr),
		\end{equation}
		since each internal stage $u_h^{(r)}|_{I_i}$ is computed locally from $u_h^n|_{I_i}$. \Cref{def:compact} also accommodates local space--time predictors obtained via cell-local implicit solves, as in ADER-DG methods. 
	\end{remark}
	
	Taking $w=1$ in \cref{eq:compactstructure} recovers the conservative cell-average update \cref{eq:mean}. Assume $f\in C^1$, $\UU\in C^2$, and uniform convexity on the state range under consideration:
	\begin{equation}\label{eq:entropycurvature}
		0<\mu\le\UU''(u)\le L.
	\end{equation}
	To compare the high-order interface flux $\hat f_{i+1/2}$ with a first-order baseline, we define the reference interface flux evaluated at the old cell averages:
	\[
	\hat g_{i+1/2}:=\hat f(\bar u_i^n,\bar u_{i+1}^n).
	\]
	Pulling $u_h^n|_{I_i}$ back to the reference interval $\widehat I=[-1,1]$ via $\hat u_i(\xi)=u_h^n(x_i+\frac h2\xi)\in\mathscr P:=\mathbb P^k(\widehat I)$, where $x_i$ is the center of $I_i$, we write the full-step interface flux in reference variables as
	\[
	\hat f_{i+1/2}=\Phi_\lambda(\hat u_i,\hat u_{i+1}),\qquad
	\Phi_\lambda:\mathscr P\times\mathscr P\longrightarrow\mathbb R.
	\]
	We require $\hat f$ and $\Phi_\lambda$ to satisfy the following three properties:
	\begin{enumerate}
		\renewcommand{\theenumi}{P\arabic{enumi}}
		\renewcommand{\labelenumi}{(\theenumi)}
		
		\item\label[operatorproperty]{cond:reference} \emph{Strict entropy dissipation.} There exists $\eta>0$ independent of the state pair such that
		\[
		[\psi]-[v]\hat g_{i+1/2}
		\ge\eta(\bar u_{i+1}^n-\bar u_i^n)^2,
		\]
		where $[v]=v(\bar u_{i+1}^n)-v(\bar u_i^n)$ and $[\psi]=\psi(\bar u_{i+1}^n)-\psi(\bar u_i^n)$.
		
		\item\label[operatorproperty]{cond:flux-lipschitz}
		\emph{Uniform Lipschitz continuity on bounded sets.}
		Fix $\lambda_*>0$. For each bounded subset of $\mathscr P$, there exists $C_\Phi\ge0$ such that for all $p,q,\tilde p,\tilde q$ in that subset and $0\le\lambda\le\lambda_*$,
		\[
		\bigl|\Phi_\lambda(p,q)-\Phi_\lambda(\tilde p,\tilde q)\bigr|
		\le C_\Phi\bigl(\|p-\tilde p\|_{L^\infty(\widehat I)}
		+\|q-\tilde q\|_{L^\infty(\widehat I)}\bigr).
		\]
		
		\item\label[operatorproperty]{cond:constant-reduction}
		\emph{Reduction on piecewise constant data.}
		For constant polynomials $a,b\in\mathscr P$, $\Phi_\lambda(a,b)=\hat f(a,b)$.
	\end{enumerate}
	
	\begin{remark}
		Property~\cref{cond:reference} strengthens Tadmor's entropy-stability condition $[\psi]-[v]\hat g_{i+1/2}\ge0$~\cite{tadmor2003entropy} to a quadratic lower bound in the state jump, which holds for standard monotone fluxes with nonvanishing numerical viscosity. For instance, the LF flux \cref{eq:LF} satisfies~\cref{cond:reference} whenever $\alpha>0$ strictly bounds the wave speed, and the modified Godunov flux $\hat f_G^{\mathrm{mod}}(a,b)=\hat f_G(a,b)-\frac{\beta}{2}(b-a)$ with $\beta>0$ satisfies~\cref{cond:reference} for $\eta=\mu\beta/2$.
		Moreover, standard SSP-RKDG schemes generally satisfy neither the
        two-cell Lipschitz property \cref{cond:flux-lipschitz} nor the
        constant-reduction property \cref{cond:constant-reduction}, because
        their internal stages couple neighboring cells.
	\end{remark}
	
\begin{proposition}[Uniform flux bounds]\label{prop:flux-perturbation}
 For reference-cell polynomials in a fixed $L^\infty$-bounded set, define
 $\delta_i^2=\sum_{\ell=1}^{N_{\mathrm{G}}}w_\ell\bigl(u_h^n(x_{i,\ell})-\bar u_i^n\bigr)^2$.
 Properties \crefrange{cond:flux-lipschitz}{cond:constant-reduction} imply that there exists $M_0\ge0$ such that
 \begin{equation}\label{eq:perturb}
 \begin{aligned}
  |\hat f_{i+1/2}-\hat g_{i+1/2}|&\le M_0(\delta_i+\delta_{i+1}),\\
  |\hat g_{i+1/2}-f(\bar u_j^n)|&\le M_0|\bar u_{i+1}^n-\bar u_i^n|,\qquad j=i,i+1.
 \end{aligned}
 \end{equation}
\end{proposition}
\begin{proof}
 Exactness of the Gauss rule on $\mathbb P^{2k}$ and norm equivalence on $\mathscr P$ give
 \[
  \delta_i^2=\frac1{|\widehat I|}\int_{\widehat I}(\hat u_i-\bar u_i^n)^2\,\mathrm{d}\xi,
  \qquad \|\hat u_i-\bar u_i^n\|_{L^\infty(\widehat I)}\le C_k\delta_i.
 \]
 By \crefrange{cond:flux-lipschitz}{cond:constant-reduction},
 \[
 \begin{aligned}
  |\hat f_{i+1/2}-\hat g_{i+1/2}|
  &\le C_\Phi\bigl(\|\hat u_i-\bar u_i^n\|_{L^\infty(\widehat I)}
       +\|\hat u_{i+1}-\bar u_{i+1}^n\|_{L^\infty(\widehat I)}\bigr)\\
  &\le C_\Phi C_k(\delta_i+\delta_{i+1}).
 \end{aligned}
 \]
 Consistency $\hat f(a,a)=f(a)$ and the same properties give
 $|\hat g_{i+1/2}-f(\bar u_j^n)|\le C_\Phi|\bar u_{i+1}^n-\bar u_i^n|$ for $j=i,i+1$.
 Taking $M_0=C_\Phi\max\{C_k,1\}$ proves both bounds.
\end{proof}

	\subsection{A reformulation of a first-order scheme}\label{sec:firstorder-feasibility}
	
	To motivate our high-order construction, we first reconsider the first-order cell-average update
	\begin{equation}\label{eq:1st}
		\bar u_i^{n+1}=\bar u_i^n
		-\lambda(\hat g_{i+1/2}-\hat g_{i-1/2}),\qquad
		\hat g_{i+1/2}=\hat f(\bar u_i^n,\bar u_{i+1}^n).
	\end{equation}
	For the LF flux \cref{eq:LF}, previous works~\cite{kivva2022entropy,liu2026limiter} employ the proper LF entropy flux
	\begin{equation}\label{eq:properLF}
		\hat{\mathcal F}_{\mathrm{LF}}(a,b)
		=\frac{\mathcal F(a)+\mathcal F(b)}2
		-\frac\alpha2\bigl(\mathcal U(b)-\mathcal U(a)\bigr)
	\end{equation}
	to establish the weak cell entropy inequality $\mathcal U(\bar u_i^{n+1})\le\mathcal U(\bar u_i^n)-\lambda(\hat{\mathcal F}_{i+1/2}-\hat{\mathcal F}_{i-1/2})$ under a CFL condition. However, \cref{eq:properLF} is only one particular choice. With the conservative update \cref{eq:1st} fixed, a natural question arises: \emph{can we construct an explicit local set of shared numerical entropy fluxes sufficient to ensure this weak cell entropy inequality?}
	
    At $x_{i+1/2}$, set $a=\bar u_i^n$, $b=\bar u_{i+1}^n$,
    and suppress the interface subscript. Define the virtual states
    \[
     Z^-=a-2\lambda\bigl(\hat g-f(a)\bigr),\qquad
     Z^+=b+2\lambda\bigl(\hat g-f(b)\bigr).
    \]
    Suppose a shared numerical entropy flux $\hat{\mathcal F}$ satisfies
    the half-cell entropy inequalities
    \begin{equation}\label{eq:half}
     \mathcal U(Z^-)
     \le \mathcal U(a)
           -2\lambda\bigl(\hat{\mathcal F}-\mathcal F(a)\bigr),\qquad 
     \mathcal U(Z^+)
     \le \mathcal U(b)
           +2\lambda\bigl(\hat{\mathcal F}-\mathcal F(b)\bigr).
    \end{equation}
    Restoring interface indices, the update \cref{eq:1st} satisfies
    \[
     \bar u_i^{n+1}
     =\frac12\bigl(Z^+_{i-1/2}+Z^-_{i+1/2}\bigr).
    \]
    Convexity and the two half-cell inequalities therefore give
    \[
     \mathcal U(\bar u_i^{n+1})
     \le \frac12\bigl(
           \mathcal U(Z^+_{i-1/2})
           +\mathcal U(Z^-_{i+1/2})\bigr)\le \mathcal U(\bar u_i^n)
           -\lambda\bigl(
           \hat{\mathcal F}_{i+1/2}
           -\hat{\mathcal F}_{i-1/2}\bigr).
    \]
    Thus, to obtain the cell entropy inequality, it suffices to enforce
    the half-cell entropy inequalities \cref{eq:half} at every interface. These are
    equivalent to
    \[
     \hat{\mathcal F}^{\mathrm{lo}}
     \le \hat{\mathcal F}
     \le \hat{\mathcal F}^{\mathrm{hi}},
    \]
    where
    \begin{equation}\label{eq:firstorder-interval}
     \hat{\mathcal F}^{\mathrm{lo}}
     =\mathcal F(b)
       +\frac{\mathcal U(Z^+)-\mathcal U(b)}{2\lambda},\qquad
     \hat{\mathcal F}^{\mathrm{hi}}
     =\mathcal F(a)
       +\frac{\mathcal U(a)-\mathcal U(Z^-)}{2\lambda}.
    \end{equation}
    We next prove that the interval $[\hat\FF^{\mathrm{lo}},\hat\FF^{\mathrm{hi}}]$ is nonempty under a suitable
    CFL condition. Suppress the interface subscript $i+1/2$ again. Taylor expansion and $\mathcal U''\le L$ give
    \[
    \begin{aligned}
     \mathcal U(Z^-)-\mathcal U(a)
     &\le-2\lambda v(a)(\hat g-f(a))+2\lambda^2L(\hat g-f(a))^2,\\
     \mathcal U(Z^+)-\mathcal U(b)
     &\le 2\lambda v(b)(\hat g-f(b))+2\lambda^2L(\hat g-f(b))^2.
    \end{aligned}
    \]
    Using $\psi=vf-\mathcal F$, property~\cref{cond:reference}, and
    $|\hat g-f(a)|,|\hat g-f(b)|\le C_\Phi|b-a|$ from \crefrange{cond:flux-lipschitz}{cond:constant-reduction}, we obtain
    \begin{equation}\label{eq:firstorder-nonempty}
    \begin{aligned}
     \hat{\mathcal F}^{\mathrm{lo}}_{i+1/2}-\hat{\mathcal F}^{\mathrm{hi}}_{i+1/2}
     &=\mathcal F(b)-\mathcal F(a)
       +\frac{\mathcal U(Z^+)-\mathcal U(b)+\mathcal U(Z^-)-\mathcal U(a)}{2\lambda}\\
     &\le [v]\hat g-[\psi]
           +\lambda L\bigl((\hat g-f(a))^2+(\hat g-f(b))^2\bigr)\\
     &\le-\bigl(\eta-2\lambda LC_\Phi^2\bigr)(b-a)^2\le0
    \end{aligned}
    \end{equation}
    whenever $2\lambda LC_\Phi^2\le\eta$. For the LF flux \cref{eq:LF}, the proper LF entropy flux \cref{eq:properLF} belongs to $[\hat\FF_{i+1/2}^\mathrm{lo},\hat\FF_{i+1/2}^{\mathrm{hi}}]$ under the standard CFL condition $2\lambda\alpha\le1$~\cite{carlier2023invariant}.
	
	This interval perspective naturally extends to high-order compact schemes. When every old DG polynomial $u_h^n|_{I_i}$ is replaced by its cell average $\bar u_i^n$, a compact single-step update reduces to \cref{eq:1st}. As the cellwise polynomial fluctuations $u_h^n|_{I_i}-\bar u_i^n$ are continuously restored, the full-step conservative flux varies continuously from $\hat g_{i+1/2}$ to $\hat f_{i+1/2}$, continuously perturbing the feasible entropy-flux intervals as well. Simultaneously, these intra-cell fluctuations introduce a nonnegative entropy gap given by the Jensen residual
	\[
	J_i^n=\widetilde{\mathcal U}_i^n-\mathcal U(\bar u_i^n)\ge0,
	\]
	which provides the additional entropy budget needed to absorb the flux perturbation $e=\hat f_{i+1/2}-\hat g_{i+1/2}$ and keep the perturbed intervals nonempty. This interpretation is a key contribution of the present work.
	
	\subsection{High-order feasible intervals}\label{sec:entropyconstruction}
	As in the first-order case, to enforce the weak cell entropy inequality
	\begin{equation}\label{eq:weakmeanentropy}
		\UU(\bar u_i^{n+1})\le\widetilde{\UU}_i^n
		-\lambda\bigl(\hat{\FF}_{i+1/2}-\hat{\FF}_{i-1/2}\bigr),
	\end{equation}
	we define the high-order virtual states
	\begin{align}\label{eq:z}
		Z_{i+1/2}^{-}&=\bar u_i^n-2\lambda\bigl(\hat f_{i+1/2}-f(\bar u_i^n)\bigr),\\
		Z_{i+1/2}^{+}&=\bar u_{i+1}^n+2\lambda\bigl(\hat f_{i+1/2}-f(\bar u_{i+1}^n)\bigr),\label{eq:z-plus}
	\end{align}
	and the corresponding feasible entropy-flux interval $[\hat{\FF}^{\mathrm{lo}}_{i+1/2},\hat{\FF}^{\mathrm{hi}}_{i+1/2}]$ with endpoints
	\begin{align}\label{eq:interval}
		\hat{\FF}^{\mathrm{lo}}_{i+1/2}&=\FF(\bar u_{i+1}^n)
		+\dfrac{\UU(Z_{i+1/2}^{+})-\widetilde{\UU}_{i+1}^n}{2\lambda},\\
		\hat{\FF}^{\mathrm{hi}}_{i+1/2}&=\FF(\bar u_i^n)
		+\dfrac{\widetilde{\UU}_i^n-\UU(Z_{i+1/2}^{-})}{2\lambda},\label{eq:interval-hi}
	\end{align}
	which reduce to \cref{eq:firstorder-interval} when $u_h^n$ is piecewise constant.
	
	\begin{theorem}[Nonempty feasible intervals]\label{thm:cfl}
		Under the assumptions of \cref{prop:flux-perturbation} and properties \crefrange{cond:reference}{cond:constant-reduction}, there exists $0<\lambda_0\le\lambda_*$, depending only on $\mu,L,\eta,M_0,\lambda_*$, such that $\hat{\FF}^{\mathrm{lo}}_{i+1/2}\le\hat{\FF}^{\mathrm{hi}}_{i+1/2}$ at every interface whenever $0<\lambda\le\lambda_0$.
	\end{theorem}
\begin{proof}
 At $x_{i+1/2}$, set $d=\bar u_{i+1}^n-\bar u_i^n$, $e=\hat f_{i+1/2}-\hat g_{i+1/2}$ and $W=\hat{\FF}^{\mathrm{hi}}_{i+1/2}-\hat{\FF}^{\mathrm{lo}}_{i+1/2}$.
 Taylor expansion as in \cref{sec:firstorder-feasibility}, \cref{cond:reference}, \cref{eq:perturb}, uniform convexity and $|[v]|\le L|d|$ give
 \[
 \begin{aligned}
 W&\ge\frac{J_i^n+J_{i+1}^n}{2\lambda}+[\psi]-[v]\hat f_{i+1/2}\\
  &\quad-\lambda L\bigl((\hat f_{i+1/2}-f(\bar u_i^n))^2
                       +(\hat f_{i+1/2}-f(\bar u_{i+1}^n))^2\bigr)\\
  &\ge\frac{J_i^n+J_{i+1}^n}{2\lambda}
       +(\eta-4\lambda LM_0^2)d^2-L|d||e|-4\lambda Le^2,\\
 J_i^n&\ge\frac\mu2\delta_i^2,\qquad
 e^2\le M_0^2(\delta_i+\delta_{i+1})^2
     \le\frac{4M_0^2}{\mu}(J_i^n+J_{i+1}^n).
 \end{aligned}
 \]
 If $4\lambda LM_0^2\le\eta/2$, Young's inequality gives
 \[
 \begin{aligned}
 L|d||e|&\le\frac\eta2d^2+\frac{L^2}{2\eta}e^2,\\
 W&\ge\frac{J_i^n+J_{i+1}^n}{2\lambda}
       -\left(\frac{L^2}{2\eta}+4\lambda L\right)e^2\\
  &\ge\left[\frac1{2\lambda}
       -\frac{4M_0^2}{\mu}\left(\frac{L^2}{2\eta}+4\lambda L\right)\right]
       (J_i^n+J_{i+1}^n)\ge0.
 \end{aligned}
 \]
 The final inequality holds for all sufficiently small $\lambda>0$, proving that every interval in \crefrange{eq:interval}{eq:interval-hi} is nonempty.
\end{proof}
	
	\begin{lemma}[Weak cell entropy inequality]\label{lem:budget}
		For any shared numerical entropy flux $\hat{\FF}_{i+1/2}\in[\hat{\FF}^{\mathrm{lo}}_{i+1/2},\hat{\FF}^{\mathrm{hi}}_{i+1/2}]$, the cell-average entropy satisfies $\UU_i^{\mathrm{1st}}\le\UU_i^{\mathrm{up}}$.
	\end{lemma}
	\begin{proof}
		Since $\bar u_i^{n+1}=(Z_{i-1/2}^{+}+Z_{i+1/2}^{-})/2$, convexity and \crefrange{eq:interval}{eq:interval-hi} give
		\[
		\UU_i^{\mathrm{1st}}
		\le\dfrac{\UU(Z_{i-1/2}^{+})+\UU(Z_{i+1/2}^{-})}{2}
		\le\widetilde{\UU}_i^n-\lambda
		\bigl(\hat{\FF}_{i+1/2}-\hat{\FF}_{i-1/2}\bigr)
		=\UU_i^{\mathrm{up}},
		\]
		which establishes the weak cell entropy inequality.
	\end{proof}
	
	With $\UU_i^{\mathrm{1st}}\le\UU_i^{\mathrm{up}}$ established, we define the terminal scaling factor~\cite{liu2026limiter} by
	\begin{equation}\label{eq:theta}
		\theta_i=\begin{cases}
			1,&\UU_i^{\mathrm{high}}\le\UU_i^{\mathrm{up}},\\
			\dfrac{\UU_i^{\mathrm{up}}-\UU_i^{\mathrm{1st}}}{\UU_i^{\mathrm{high}}-\UU_i^{\mathrm{1st}}},
			&\text{otherwise}.
		\end{cases}
	\end{equation}
	
	\begin{lemma}[Strong cell entropy inequality \cite{liu2026limiter}]\label{thm:es}
		Suppose \cref{lem:budget} holds. Then the scaled update \cref{eq:scalingprelim} with $\theta_i$ given by \cref{eq:theta} preserves $\bar u_i^{n+1}$ and satisfies
		\begin{equation}\label{eq:es}
			\widetilde{\UU}_i^{n+1}\le \widetilde{\UU}_i^n-
			\lambda(\hat{\FF}_{i+1/2}-\hat{\FF}_{i-1/2}).
		\end{equation}
	\end{lemma}
	
	The framework extends directly to multiple entropy pairs by taking the minimum of their corresponding scaling factors.
    
	\begin{corollary}[Multiple scalar entropies]
		Given a finite family of convex entropy pairs with feasible budgets, setting $\theta_i$ to the minimum of their individual scaling factors in \cref{eq:theta} enforces the strong cell entropy inequality \cref{eq:es} for all pairs simultaneously.
	\end{corollary}
	
	Following the standard Lax--Wendroff argument in~\cite{liu2026limiter}, the cellwise conservative update \cref{eq:mean} and strong cell entropy inequality \cref{eq:es} pass to strong $L^1$ limits. We omit the proof for brevity.
	
	\begin{theorem}[Lax--Wendroff-type convergence]\label{thm:laxwendroff}
		On a periodic domain $\Omega$, fix $k$ and let $h,\Delta t\to0$ with $0<\Delta t/h\le\lambda_*$. Suppose \crefrange{cond:flux-lipschitz}{cond:constant-reduction} and the hypotheses of \cref{thm:es} hold at every step, and assume that the shared numerical entropy flux satisfies
		\[
		\bigl|\hat{\FF}_{i+1/2}-\FF(p)\bigr|
		\le C\left\|u_h^n-p\right\|_{L^\infty(I_i\cup I_{i+1})}
		\]
		for numerical data and constant polynomials $p\in\mathscr P$ in any fixed bounded range, with $C$ independent of $i$ and $n$. Set $u_h(x,t)=u_h^n(x)$ for $t_n\le t<t_{n+1}$, and assume $u_h$ is uniformly bounded on $\Omega\times[0,T]$ for every $T>0$. If
		\[
		u_h\to u\quad\text{in }L^1(\Omega\times(0,T)),\qquad
		u_h^0\to u_0\quad\text{in }L^1(\Omega)
		\]
		for every $T>0$, then $u$ is a weak solution of \cref{eq:scalarprelim} satisfying \cref{eq:weakentropy} for $(\UU,\FF)$. In particular, the conclusion holds simultaneously for any finite family of enforced entropy pairs.
	\end{theorem}
	
	\subsection{High-order entropy-flux design and accuracy preservation}\label{sec:accuracy}
	To maintain high-order accuracy in smooth regions while guaranteeing entropy stability, we define $\hat{\FF}_{i+1/2}$ by projecting a high-order target entropy flux $\hat{\FF}^{\tar}_{i+1/2}$ onto the feasible interval:
	\begin{equation}\label{eq:projection}
		\hat{\FF}_{i+1/2}=\min\{\hat{\FF}^{\mathrm{hi}}_{i+1/2},
		\max\{\hat{\FF}^{\mathrm{lo}}_{i+1/2},\hat{\FF}^{\tar}_{i+1/2}\}\}.
	\end{equation}
	On any cell where the limiter is active with $\theta_i<1$, \cref{eq:scalingprelim} and \cref{eq:theta} give
	\begin{equation}\label{eq:perturbsize}
		\|u_h^{n+1}-u_h^{n+1,(\pre)}\|_\infty
		=\dfrac{\UU_i^{\mathrm{high}}-\UU_i^{\mathrm{up}}}{\UU_i^{\mathrm{high}}-\UU_i^{\mathrm{1st}}}
		\|u_h^{n+1,(\pre)}-\bar u_i^{n+1}\|_\infty,
	\end{equation}
	where cellwise $L^\infty$ norms are taken over $I_i$. Set
	\[
	\begin{aligned}
		D_i&=(\UU_i^{\mathrm{high}}-\UU_i^{\mathrm{up}})_+,\qquad
		J_i^*=\UU_i^{\mathrm{high}}-\UU_i^{\mathrm{1st}},\\
		A_i^m&=h|u_x(x_i,t_m)|+h^2,\qquad m=n,n+1.
	\end{aligned}
	\]
	The local variation scale $A_i^n$ is $O(h)$ away from critical points and $O(h^2)$ within $O(h)$ of a nondegenerate extremum. Fix $k\ge2$ and assume \crefrange{cond:flux-lipschitz}{cond:constant-reduction}.
	
	\begin{assumption}[Smooth solution and approximation accuracy]\label{ass:smooth-patch}
		On a fixed space--time neighborhood containing the stencils and one-step characteristics of the considered cells, we assume the following for $0<\lambda\le\bar\lambda\le\lambda_*$:
		\begin{enumerate}
			\item\label{ass:smooth-regularity}
			The exact solution $u$ is smooth, all required derivatives of $u,f,\UU$ are uniformly bounded, and the nondegeneracy condition $|u_x|+|u_{xx}|\ge c_*>0$ holds.
			\item\label{ass:old-accuracy}
			The old solution satisfies $\|u_h^n-u(\cdot,t_n)\|_\infty\le Ch^{k+1}$.
			\item\label{ass:candidate-accuracy}
			The preliminary solution satisfies $\|u_h^{n+1,(\pre)}-u(\cdot,t_{n+1})\|_\infty\le Ch^{k+1}$.
		\end{enumerate}
	\end{assumption}
	
	\begin{assumption}[Nonlinear consistency of the entropy-flux target]\label{ass:entropy-target}
		For each cell $I_i$ in the smooth patch and $e=i\pm1/2$, the target entropy flux satisfies
		\begin{equation}\label{eq:weightedconsistency}
			\left|\hat{\FF}^{\tar}_e-\UU'(\bar u_i^n)\hat f_e
			-\dfrac1{\Delta t}\int_{t_n}^{t_{n+1}}
			\bigl[\FF(u(x_e,t))-\UU'(\bar u_i^n)f(u(x_e,t))\bigr]\,\mathrm{d}t
			\right|\le C A_i^n h^{k+1}.
		\end{equation}
	\end{assumption}
	
	Taylor expansion, uniform convexity \cref{eq:entropycurvature}, and the nondegeneracy condition $|u_x|+|u_{xx}|\ge c_*>0$ in \cref{ass:smooth-patch} directly yield the following scaling estimates for the local fluctuations and Jensen residuals.
	
	\begin{lemma}\label{lem:accuracy-amplitude}
		Under \cref{ass:smooth-patch}, for sufficiently small $h$, there exist $c,c_J>0$ such that:
		\begin{enumerate}
			\renewcommand{\labelenumi}{(\arabic{enumi})}
			\item $cA_i^n\le A_i^{n+1}\le CA_i^n$.
			\item $\|u_h^{n+1,(\pre)}-\bar u_i^{n+1}\|_\infty\le CA_i^n$ and $\|u_h^n-\bar u_i^n\|_\infty\le CA_i^n$.
			\item The Jensen residuals satisfy
			\begin{equation}\label{eq:amplitudejensen}
				c_J (A_i^n)^2\le J_i^n,J_i^*\le C (A_i^n)^2.
			\end{equation}
		\end{enumerate}
	\end{lemma}
    
	In view of part~(1), we write $A_i=A_i^n$ below. Next, we show that in smooth regions under a suitable CFL condition, the high-order target flux $\hat{\FF}^{\tar}_{i+1/2}$ already lies inside the feasible interval, so the projection \cref{eq:projection} is inactive.
	
	\begin{lemma}\label{lem:accuracy-target-feasible}
		Under \cref{ass:smooth-patch,ass:entropy-target}, there exists $\bar\lambda>0$ such that for sufficiently small $h$ and $0<\lambda\le\bar\lambda$,
		\[
		\hat{\FF}^{\mathrm{lo}}_{i+1/2}
		\le\hat{\FF}^{\tar}_{i+1/2}\le\hat{\FF}^{\mathrm{hi}}_{i+1/2}.
		\]
	\end{lemma}
\begin{proof}
\cref{lem:accuracy-amplitude} and \crefrange{cond:flux-lipschitz}{cond:constant-reduction} give $|\hat f_{i+1/2}-f(\bar u_i^n)|\le CA_i$ and $|Z_{i+1/2}^--\bar u_i^n|\le C\lambda A_i$.
With $u=u(x_{i+1/2},t)$, entropy compatibility and \cref{eq:weightedconsistency} yield
\[
\begin{aligned}
 &\left|\hat{\FF}^{\tar}_{i+1/2}-\FF(\bar u_i^n)
       -\UU'(\bar u_i^n)\bigl(\hat f_{i+1/2}-f(\bar u_i^n)\bigr)\right|\\
 &\le CA_i h^{k+1}
  +\frac1{\Delta t}\int_{t_n}^{t_{n+1}}
    \left|\int_{\bar u_i^n}^{u}
       \bigl[\UU'(z)-\UU'(\bar u_i^n)\bigr]f'(z)\,\mathrm{d}z\right|\,\mathrm{d}t\\
 &\le CA_i h^{k+1}
  +\frac C{\Delta t}\int_{t_n}^{t_{n+1}}|u-\bar u_i^n|^2\,\mathrm{d}t
 \le CA_i^2.
\end{aligned}
\]
Taylor expansion, \cref{eq:interval-hi} and \cref{eq:amplitudejensen} now give
\[
\begin{aligned}
 0&\le\UU(Z_{i+1/2}^-)-\UU(\bar u_i^n)
       -\UU'(\bar u_i^n)(Z_{i+1/2}^--\bar u_i^n)\\
  &\le\frac L2|Z_{i+1/2}^--\bar u_i^n|^2\le C\lambda^2 A_i^2,\\
 \hat{\FF}^{\mathrm{hi}}_{i+1/2}-\hat{\FF}^{\tar}_{i+1/2}
  &\ge\frac{J_i^n}{2\lambda}-CA_i^2-C\lambda A_i^2
   \ge A_i^2\left(\frac{c_J}{2\lambda}-C(1+\lambda)\right)\ge0
\end{aligned}
\]
for sufficiently small $\bar\lambda>0$.
The same expansion about $\bar u_{i+1}^n$ gives
$\hat{\FF}^{\tar}_{i+1/2}-\hat{\FF}^{\mathrm{lo}}_{i+1/2}\ge A_{i+1}^2\bigl(\frac{c_J}{2\lambda}-C(1+\lambda)\bigr)\ge0$, proving the lemma.
\end{proof}
	
	\begin{theorem}[Accuracy preservation]\label{thm:localaccuracy}
		Under \cref{ass:smooth-patch,ass:entropy-target}, let $\hat{\FF}_{i+1/2}$ be selected by \cref{eq:projection}. For sufficiently small $h$ and $0<\lambda\le\bar\lambda$,
		\begin{equation}\label{eq:uniformlocalaccuracy}
			D_i\le C A_i h^{k+1},\qquad
			\|u_h^{n+1}-u_h^{n+1,(\pre)}\|_\infty\le C h^{k+1},
		\end{equation}
		uniformly across the smooth patch, including cells near nondegenerate extrema.
	\end{theorem}
\begin{proof}
By \cref{lem:accuracy-target-feasible}, $\hat{\FF}_e=\hat{\FF}^{\tar}_e$ at $e=i\pm1/2$, so $D_i=[\UU_i^{\mathrm{high}}-\widetilde{\UU}_i^n+\lambda(\hat{\FF}^{\tar}_{i+1/2}-\hat{\FF}^{\tar}_{i-1/2})]_+$. Write $u^m(x)=u(x,t_m)$ and $u_e(t)=u(x_e,t)$. Subtracting the exact cell entropy balance
\[
 \frac1h\int_{I_i}[\UU(u^{n+1})-\UU(u^n)]\,\mathrm dx
 +\frac1h\int_{t_n}^{t_{n+1}}[\FF(u_e)]_{e=i-1/2}^{e=i+1/2}\,\mathrm dt=0
\]
yields the decomposition
\[
\begin{aligned}
 \UU_i^{\mathrm{high}}-\UU_i^{\mathrm{up}}
 ={}&\left[\UU_i^{\mathrm{high}}-\frac1h\int_{I_i}\UU(u^{n+1})\,\mathrm dx\right]
 -\left[\widetilde{\UU}_i^n-\frac1h\int_{I_i}\UU(u^n)\,\mathrm dx\right]\\
 &+\lambda\left[\hat{\FF}^{\tar}_e-\frac1{\Delta t}
 \int_{t_n}^{t_{n+1}}\FF(u_e)\,\mathrm dt\right]_{e=i-1/2}^{e=i+1/2}.
\end{aligned}
\]
At each Gauss node $x_{i,\ell}$, setting $a_\ell=u_h^{n+1,(\pre)}(x_{i,\ell})$ and $b_\ell=u^{n+1}(x_{i,\ell})$, we have
\[
 \UU(a_\ell)-\UU(b_\ell)-\UU'(\bar u_i^n)(a_\ell-b_\ell)
 =\int_{b_\ell}^{a_\ell}[\UU'(z)-\UU'(\bar u_i^n)]\,\mathrm dz
 =O(A_i h^{k+1}).
\]
Since the $N_{\mathrm G}$-point Gauss rule is exact through degree $2k\ge k+2$, its quadrature error on $u$ and $\UU(u)$ is $O(h^{k+3})$. Summing over $\ell$ gives the new-time estimate and, similarly, the old-time estimate:
\begin{align}
 \UU_i^{\mathrm{high}}-\frac1h\int_{I_i}\UU(u^{n+1})\,\mathrm dx
 & =\UU'(\bar u_i^n)\sum_\ell w_\ell(a_\ell-b_\ell)
 +O(A_i h^{k+1}+h^{k+3})\notag\\
 & =\UU'(\bar u_i^n)\left[\bar u_i^{n+1}
 -\frac1h\int_{I_i}u^{n+1}\,\mathrm dx\right]
 +O(A_i h^{k+1}+h^{k+3}),\label{eq:newentropy-error}\\
 \widetilde{\UU}_i^n-\frac1h\int_{I_i}\UU(u^n)\,\mathrm dx
 & =\UU'(\bar u_i^n)\left[\bar u_i^n
 -\frac1h\int_{I_i}u^n\,\mathrm dx\right]
 +O(A_i h^{k+1}+h^{k+3}).\label{eq:oldentropy-error}
\end{align}
At $e=i\pm1/2$, \cref{ass:entropy-target} gives
\begin{equation}\label{eq:faceentropy-error}
 \hat{\FF}^{\tar}_e-\frac1{\Delta t}\int_{t_n}^{t_{n+1}}\FF(u_e)\,\mathrm dt
 =\UU'(\bar u_i^n)\left[\hat f_e-\frac1{\Delta t}
 \int_{t_n}^{t_{n+1}}f(u_e)\,\mathrm dt\right]+O(A_i h^{k+1}).
\end{equation}
By the conservative cell-average update \cref{eq:mean} and the exact conservation law, the leading $\UU'(\bar u_i^n)$-weighted terms satisfy
\begin{equation}\label{eq:conservative-error-balance}
\begin{aligned}
 \bar u_i^{n+1}-\bar u_i^n-\frac1h\int_{I_i}(u^{n+1}-u^n)\,\mathrm dx
 =-\lambda\left[\hat f_e-\frac1{\Delta t}
 \int_{t_n}^{t_{n+1}}f(u_e)\,\mathrm dt\right]_{e=i-1/2}^{e=i+1/2}
\end{aligned}
\end{equation}
and therefore cancel identically when \crefrange{eq:newentropy-error}{eq:faceentropy-error} are combined. Since $h^2\le A_i$, we obtain
\[
 D_i\le|\UU_i^{\mathrm{high}}-\UU_i^{\mathrm{up}}|
 \le CA_i h^{k+1}+Ch^{k+3}+C\lambda A_i h^{k+1}
 \le CA_i h^{k+1}.
\]
Substituting this bound and \cref{lem:accuracy-amplitude} into \cref{eq:perturbsize} yields
\[
 \|u_h^{n+1}-u_h^{n+1,(\pre)}\|_\infty
 \le\frac{D_i}{c_J A_i^2}CA_i\le Ch^{k+1},
\]
completing the proof.
\end{proof}
	
	For schemes based on the LF conservative flux \cref{eq:LF} and proper LF entropy flux \cref{eq:properLF}, we propose a unified target construction. Let $p_i(x,s)$ be a local predictor approximating $u(x,t_n+s\Delta t)$ for $s\in[0,1]$, with interface traces $p^-(s)=p_i(x_{i+1/2},s)$ and $p^+(s)=p_{i+1}(x_{i+1/2},s)$, and define
	\begin{equation}\label{eq:unifiedtarget}
		\begin{aligned}
			\hat{\FF}_{i+1/2}^{\tar}
			={}&\int_0^1\hat{\FF}_{\mathrm{LF}}(p^-(s),p^+(s))\,\mathrm{d}s\\
			&+\{\UU'(u_h^n)\}\left(
			\hat f_{i+1/2}-\int_0^1\hat f_{\mathrm{LF}}(p^-(s),p^+(s))\,\mathrm{d}s\right).
		\end{aligned}
	\end{equation}
	
	\begin{proposition}\label{prop:unifiedtarget}
		Assume the interface-flux accuracy in \cref{sec:compactstructure} and parts~(1)--(2) of \cref{ass:smooth-patch}. If $w(s)=u(x_{i+1/2},t_n+s\Delta t)$ and the predictor traces satisfy
		\begin{equation}\label{eq:predictormatching}
			\sup_{0\le s\le1}
			\bigl(|p^-(s)-w(s)|+|p^+(s)-w(s)|\bigr)\le Ch^{k+1},
		\end{equation}
		then the unified target \cref{eq:unifiedtarget} satisfies \cref{ass:entropy-target}.
	\end{proposition}
	\begin{proof}
		Lipschitz continuity and consistency of $\hat f_{\mathrm{LF}}$, \cref{eq:predictormatching}, and the interface-flux accuracy give
		\[
		\left|\hat f_{i+1/2}
		-\int_0^1\hat f_{\mathrm{LF}}(p^-(s),p^+(s))\,\mathrm{d}s\right|
		\le Ch^{k+1}.
		\]
		Expanding \cref{eq:unifiedtarget} about $\bar u_i^n$ gives
		\begin{align*}
			&\hat{\FF}_{i+1/2}^{\tar}-\UU'(\bar u_i^n)\hat f_{i+1/2}-\int_0^1\bigl(\FF(w(s))-\UU'(\bar u_i^n)f(w(s))\bigr)\,\mathrm{d}s\\
			&=\frac12\sum_{\sigma\in\{-,+\}}\int_0^1
			\Bigl(\FF(p^\sigma(s))-\FF(w(s))-\UU'(\bar u_i^n)\bigl(f(p^\sigma(s))-f(w(s))\bigr)\Bigr)\,\mathrm{d}s\\
			&\quad-\frac\alpha2\int_0^1
			\Bigl(\UU(p^+(s))-\UU(p^-(s))-\UU'(\bar u_i^n)\bigl(p^+(s)-p^-(s)\bigr)\Bigr)\,\mathrm{d}s\\
			&\quad+\bigl(\{\UU'(u_h^n)\}-\UU'(\bar u_i^n)\bigr)\left(\hat f_{i+1/2}
			-\int_0^1\hat f_{\mathrm{LF}}(p^-(s),p^+(s))\,\mathrm{d}s\right).
		\end{align*}
		Since $\FF'=\UU'f'$, all states $z$ on the segments between $w(s)$, $p^\pm(s)$, and $\bar u_i^n$ lie within $O(A_i)$ of $\bar u_i^n$ and satisfy $|\FF'(z)-\UU'(\bar u_i^n)f'(z)|+|\UU'(z)-\UU'(\bar u_i^n)|\le CA_i$ as well as $|\{\UU'(u_h^n)\}-\UU'(\bar u_i^n)|\le CA_i$. Applying the mean value theorem and \cref{eq:predictormatching} bounds each of the three terms on the right-hand side by $CA_i h^{k+1}$, proving \cref{eq:weightedconsistency} at $x_{i+1/2}$ and analogously at $x_{i-1/2}$.
	\end{proof}
	
	\begin{remark}[Time quadrature]\label{rem:target-time-quadrature}
    In practice, we use numerical quadrature to evaluate the time integrals.
    It can be shown that \cref{prop:unifiedtarget} remains valid
    when both time integrals in \cref{eq:unifiedtarget} are evaluated by
    a common positive quadrature rule on $[0,1]$ exact through degree $k+2$.
    \end{remark}
	
\subsection{Examples of compact single-step DG methods}\label{sec:realizations}
In this section, we present several compact single-step DG methods
as examples of the proposed framework. Specifically, we consider
cRKDG, LWDG, and ADER-DG. With suitable parameter choices, their
full-step flux operators satisfy
\crefrange{cond:flux-lipschitz}{cond:constant-reduction}
on bounded state sets. We emphasize here that the underlying DG schemes for computing the preliminary solution remain unchanged; we focus only on constructing
the high-order target entropy flux $\hat{\mathcal F}^{\mathrm{tar}}$ or the predictor $p_i$.

\subsubsection{Compact RKDG}\label{sec:crkrealization}
Consider the cRKDG method of \cref{sec:crkstructure} with an $s_{\mathrm{RK}}$-stage explicit RK tableau $(a_{rj},b_r)$ of order $k+1$ and fluxes \cref{eq:cRKflux}. 

We first give a unified construction valid for arbitrary $k$. For \cref{eq:unifiedtarget}, take $(m+1)$ Gauss--Legendre time nodes $\{s_q\}_{q=1}^{m+1}\subset[0,1]$ with $m\ge k$, giving exactness through degree $2m+1\ge k+2$, and apply the local RK update over $s_q\Delta t$:
\begin{equation}\label{eq:crkgeneralpredictor}
\begin{aligned}
 Y_i^{(r)}&=u_h^n|_{I_i}+s_q\Delta t
 \sum_{j<r}a_{rj}\mathcal L^{\loc}Y_i^{(j)},\\
 p_i(\cdot,s_q)&=u_h^n|_{I_i}+s_q\Delta t
 \sum_{r=1}^{s_{\mathrm{RK}}}b_r\mathcal L^{\loc}Y_i^{(r)}.
\end{aligned}
\end{equation}
Let $p_i(x,s)$ be the degree-$m$ temporal interpolant of these nodal values.
This predictor satisfies \cref{eq:predictormatching}, and the full-step conservative flux satisfies the interface-flux accuracy requirement in \cref{sec:compactstructure}.

However, repeated fractional-step RK updates increase cost. Alternatively, a continuous extension can reuse the full-step local stage derivatives $K_i^{(r)}=\mathcal L^{\mathrm{loc}}u_h^{(r)}$. For $k=2$ and Heun's third-order tableau,
\[
 p_i(x,s)=u_h^n(x)+\Delta t\left[
 sK_i^{(1)}(x)+\frac32s^2(K_i^{(2)}(x)-K_i^{(1)}(x))\right]
\]
uses only the first two derivatives, requires no extra operator evaluations, and satisfies \cref{eq:predictormatching}. Other degrees require a case-by-case construction.

\begin{remark}\label{rem:rk-averaged-flux}
Sun~\cite{sun2026fully} uses $\hat{\FF}^{\mathrm{RK}}_{i+1/2}=\sum_{\ell=1}^{s_{\mathrm{RK}}}b_\ell\hat{\FF}_{\mathrm{LF}}(u_h^{n,(\ell),-},u_h^{n,(\ell),+})$ as a stage-averaged entropy flux, which can reduce accuracy near smooth extrema. For $u_t+u_x=0$, $\UU(u)=\FF(u)=u^2/2$, and $u_h^0(x)=u_0(x)=x^2$, the preliminary $\mathbb P^2$ cRKDG update with Heun's third-order RK method exactly reproduces $(x-\Delta t)^2$ on $I_i=[-h/2,h/2]$, whereas $\hat{\FF}^{\mathrm{RK}}_{i\pm1/2}$ gives $D_i=\Delta t^4/18$. With $\lambda=\Delta t/h$, \cref{eq:theta} yields
\[
 1-\theta_i=\frac{20\lambda^4}{1+60\lambda^2},\qquad
 u_h^1(0)-u(0,\Delta t)=\frac{5\lambda^4}{3(1+60\lambda^2)}h^2.
\]
Thus single-step accuracy is reduced to $O(h^2)$ for any fixed $\lambda>0$. Our unified target \cref{eq:unifiedtarget} instead gives $D_i=0$ and $\theta_i=1$. However, no order reduction is observed in their numerical experiments.
This may be due to the small CFL number of 0.01 used in the accuracy tests,
which can mask the lower-order correction error over the tested
mesh range. 
\end{remark}

\subsubsection{Lax--Wendroff DG}
The LWDG method~\cite{qiu2005} uses the Cauchy--Kowalevski relation $u_t=-f'(u)u_x$ to replace temporal Taylor coefficients by spatial derivatives within each cell. The Taylor polynomial $\sum_{m=0}^{k+1}\Delta t^m\partial_t^m u(x,t)/m!$ yields the conservative form \cref{eq:compactstructure} with the volume flux, LF interface flux, and degree-$k$ predictor
\begin{align}
 f_{h,\mathrm{LW}}^\star&=\sum_{m=0}^{k}\frac{\Delta t^m}{(m+1)!}\partial_t^m f(u_h^n),\label{eq:lwflux}\\
 \hat f_{i+1/2}^{\mathrm{LW}}&=\{f_{h,\mathrm{LW}}^\star\}_{i+1/2}
 -\frac{\alpha_{i+1/2}}2[u_h^n]_{i+1/2},\notag\\
 p_i(x,s)&=\sum_{m=0}^{k}\frac{(s\Delta t)^m}{m!}\partial_t^m u_h^n(x).\label{eq:lwmassface}
\end{align}
This construction is straightforward and satisfies \cref{eq:predictormatching} and the interface-flux accuracy requirement in \cref{sec:compactstructure}.

\subsubsection{ADER-DG}
The ADER-DG~\cite{dumbser2008,busto2020} method determines a space--time predictor $p_i\in\mathbb P^k(I_i)\otimes\mathbb P^k([0,1])$ independently within each cell. With normalized time $s=(t-t_n)/\Delta t$, its local Galerkin formulation is
\begin{equation}\label{eq:ADER}
 \int_0^1\!\int_{I_i}\vartheta
 [\partial_s p_i+\Delta t\,\partial_x f(p_i)]\,\mathrm dx\,\mathrm ds
 +\int_{I_i}\vartheta(x,0)[p_i(x,0)-u_h^n]\,\mathrm dx=0
\end{equation}
for every test function $\vartheta$ in the same space.  The preliminary update \cref{eq:compactstructure} uses the corrector fluxes
\begin{equation}\label{eq:aderconservativefluxes}
 f_{h,\mathrm{ADER}}^\star(x)|_{I_i}=\int_0^1 f(p_i(x,s))\,\mathrm ds,\qquad
 \hat f_{i+1/2}^{\mathrm{ADER}}=\int_0^1\hat f_{\mathrm{LF}}(p^-(s),p^+(s))\,\mathrm ds.
\end{equation}
Under parts~(1)--(2) of \cref{ass:smooth-patch}, the exact small-step solution branch of \cref{eq:ADER} satisfies \cref{eq:predictormatching} and the interface-flux accuracy required in \cref{sec:compactstructure}.
Using the same positive time quadrature, exact through degree $k+2$, in \cref{eq:aderconservativefluxes} and \cref{eq:unifiedtarget} makes the conservative-flux matching term vanish identically, so $\hat{\FF}_{i+1/2}^{\tar}=\int_0^1\hat{\FF}_{\mathrm{LF}}(p^-(s),p^+(s))\,\mathrm ds$. \Cref{prop:unifiedtarget} therefore yields \cref{ass:entropy-target}.

\subsection{Extensions to systems and multidimensional problems}\label{sec:extensions}
	For a system of conservation laws \cref{eq:systemprelim} equipped with a strictly convex entropy $\UU$ satisfying $\mu I\preceq \mathcal U''(\mathbf u)\preceq LI$, the conservative fluxes $\hat{\mathbf f}_{i+1/2}$ and virtual states $\mathbf Z_{i+1/2}^\pm$ are vector-valued, whereas the numerical entropy fluxes and their feasible intervals $[\hat{\FF}_{i+1/2}^{\mathrm{lo}},\hat{\FF}_{i+1/2}^{\mathrm{hi}}]$ remain scalar. We additionally assume that all states entering entropy evaluations, including the virtual states, lie in a convex admissible range on which the stated Hessian bounds hold. Under the vector counterparts of \crefrange{cond:reference}{cond:constant-reduction}, replacing scalar squares with squared Euclidean norms in the proof of \cref{thm:cfl} establishes nonempty feasible intervals under a CFL restriction.
	
	On $d$-dimensional Cartesian tensor-product meshes, the volume predictor remains cell-local and each face flux depends only on the two adjacent cells. Writing the cell-average update as an equal-weight convex combination of $2d$ directional virtual states with weight $1/(2d)$ yields feasible entropy-flux intervals on each face under a dimension-dependent CFL condition. Evaluating the discrete cell entropy via a $(k+1)^d$-point tensor-product Gauss--Legendre quadrature rule, we apply a single scalar scaling factor $\theta_K\in[0,1]$ to all state components on cell $K$:
	\[
	\mathbf u_h^{n+1}|_K
	=\bar{\mathbf u}_K^{n+1}
	+\theta_K\bigl(
	\mathbf u_h^{n+1,(\mathrm{pre})}|_K
	-\bar{\mathbf u}_K^{n+1}\bigr),
	\]
	with $\theta_K$ defined as in \cref{eq:theta} from the multidimensional cell entropy budget, thereby preserving the vector cell average $\bar{\mathbf u}_K^{n+1}$ and enforcing the strong cell entropy inequality.

	\section{Numerical experiments}\label{sec:experiments}
	We now present numerical experiments on uniform meshes. For each $k$, the temporal discretization has the same order of accuracy as the spatial discretization. The scalar accuracy tests with $k=2,3,4$ use CFL numbers $0.05,0.02,0.01$. Unless otherwise noted, the remaining tests use the $\mathbb P^2$ ADER-DG scheme and no limiter beyond the terminal entropy scaling \cref{eq:scalingprelim}. We denote the entropy-stabilized scheme by ``ES'', the underlying scheme without terminal scaling by ``baseline'', and the cellwise entropy residual by $R_i^n=\widetilde\UU_i^{n+1}-\UU_i^{\mathrm{up}}$, so that the strong cell entropy inequality \cref{eq:es} holds if and only if $R_i^n\le0$.

	\subsection{Scalar problems}
	\begin{Ex}[Linear advection]\label{sec:linear-test}
		Consider $u_t+u_x=0$ on the periodic interval $[0,2\pi]$ with $u(x,0)=\sin^4x$ and exponential entropy $\UU(u)=e^u$. \Cref{fig:linear-convergence} shows the $L^2$ errors at $T=1$ for cRKDG, LWDG, and ADER-DG with $k=2,3,4$. The observed convergence rates are optimal for all three schemes.

		\begin{figure}[!htbp]
\centering\includegraphics[width=.96\textwidth]{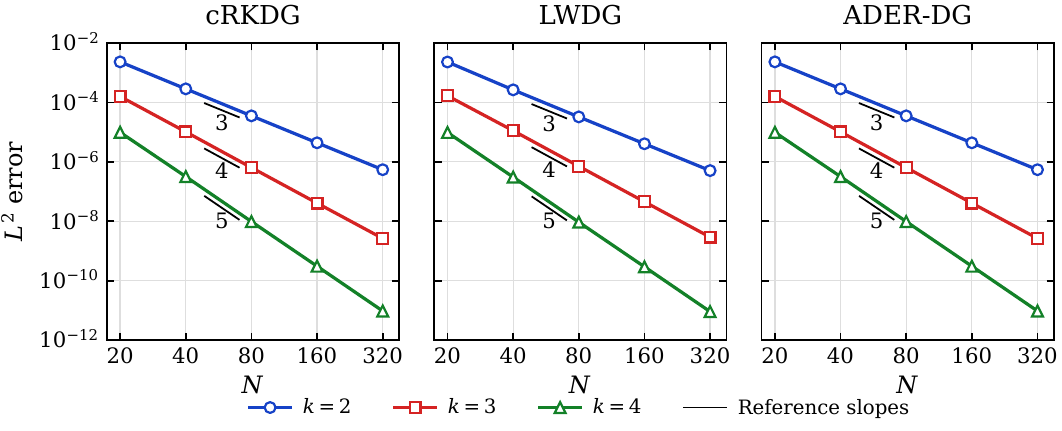}
\caption{\Cref{sec:linear-test}: Linear advection. $L^2$ errors with the exponential entropy at $T=1$. The black segments indicate the labeled reference orders.}
\label{fig:linear-convergence}
\end{figure}
		
		Next, we transport the discontinuous initial data $u(x,0)=\chi_{(\pi/2,3\pi/2)}(x)$ through one full period using $\mathbb P^2$ ADER-DG without a bound-preserving limiter. To examine how multi-entropy enforcement suppresses oscillations~\cite{chen2017entropy,liu2024entropy,liu2026limiter}, we impose the shifted arctan entropies
		\begin{equation}\label{eq:arctan-test-entropy}
\mathcal A_{a,b}(u)=\int_b^u\arctan(a(s-b))\,\mathrm{d}s,\qquad a>0,
\end{equation}
comparing the square entropy, a single shift $b=0.5$, and five or nine shifts uniformly spaced across $[0,1]$. As shown in \cref{fig:arctan-linear} at $T=2\pi$ on $N=160$ cells, enforcing five shifts simultaneously substantially reduces overshoots near the discontinuities compared with a single entropy, while nine shifts with a steeper parameter $a$ further damp oscillations at the cost of a slightly broader transition.

		\begin{figure}[!htbp]
			\centering\includegraphics[width=.96\textwidth]{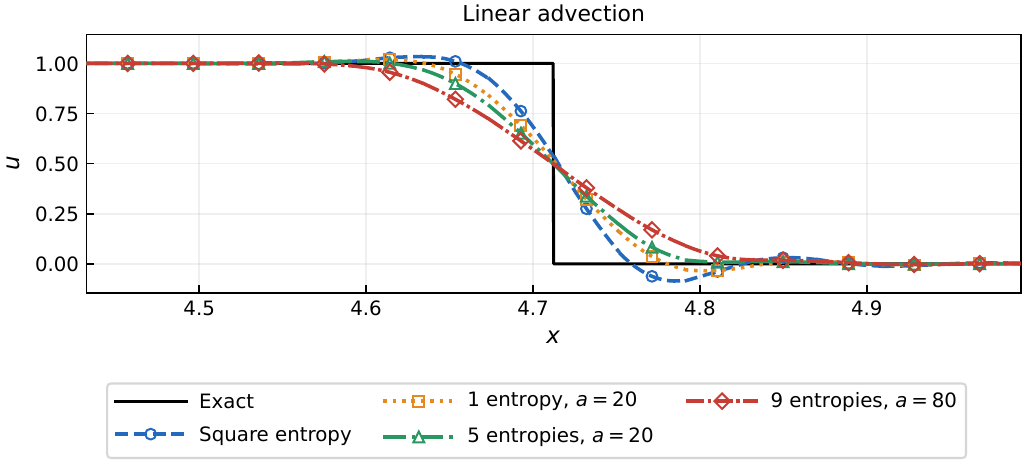}
			\caption{\Cref{sec:linear-test}: Linear advection. Numerical profiles near the discontinuity with simultaneous arctan entropy constraints at $T=2\pi$, $N=160$.}\label{fig:arctan-linear}
		\end{figure}
	\end{Ex}
	
	\begin{Ex}[Burgers' equation]\label{sec:burgers-test}
		Consider $u_t+(u^2/2)_x=0$ on the periodic interval $[0,2\pi]$ with $u(x,0)=0.5+\sin x$, which remains smooth until $T_{\mathrm{s}}=1$. With the exponential entropy $\UU(u)=e^u$, \cref{fig:burgers-convergence} reports the $L^2$ errors at $T=0.6$ for cRKDG, LWDG, and ADER-DG with $k=2,3,4$. All three methods achieve the expected convergence orders.

		\begin{figure}[!htbp]
\centering\includegraphics[width=.96\textwidth]{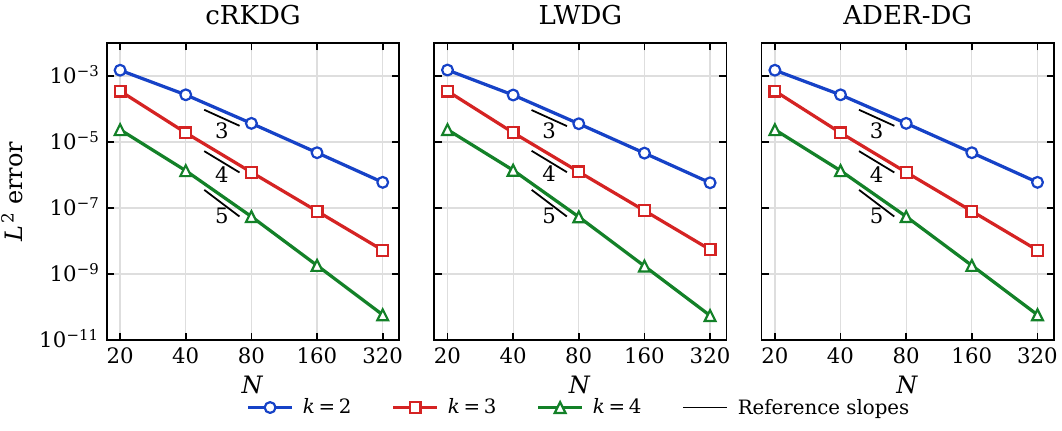}
\caption{\Cref{sec:burgers-test}: Burgers' equation. $L^2$ errors with the exponential entropy at $T=0.6$. The black segments indicate the labeled reference orders.}
\label{fig:burgers-convergence}
\end{figure}

	\end{Ex}
	
	\begin{Ex}[Buckley--Leverett equation]\label{ex:bl}
		Consider the nonconvex Buckley--Leverett flux
		\[
		f(u)=\dfrac{4u^2}{4u^2+(1-u)^2}
		\]
		on $[-0.5,0.5]$ with two Riemann problems at $x=0$: Data~1 with $(u_{\mathrm{L}},u_{\mathrm{R}})=(-3,3)$ and Data~2 with $(u_{\mathrm{L}},u_{\mathrm{R}})=(2,-2)$. Following~\cite{chen2017entropy, liu2024entropy, liu2026limiter}, we select two arctan entropies from \cref{eq:arctan-test-entropy},
		\[
		\UU^{(1)}(u)=\mathcal A_{20,0}(u),\qquad \UU^{(2)}(u)=\mathcal A_{20,1}(u),
		\]
		and test them individually and simultaneously at $T=1$ on $N=160$ cells, combined with a terminal bound-preserving limiter~\cite{zhang2012maximum}.
		
		As shown in \cref{fig:bl} and \cref{tab:bl-results}, enforcing $\UU^{(2)}$ alone produces a nonphysical intermediate plateau for Data~1, whereas enforcing $\UU^{(1)}$ alone leaves large oscillations for Data~2; in each single-entropy run, the unenforced entropy exhibits positive cell residuals. Enforcing $\UU^{(1)}$ and $\UU^{(2)}$ simultaneously satisfies both cell entropy inequalities to machine precision and captures the physically admissible composite wave structure for both Riemann problems. This highlights the importance of simultaneous entropy control in selecting physically admissible solutions for nonconvex conservation laws.
		
		\begin{figure}[!htbp]
			\centering\includegraphics[width=.96\textwidth]{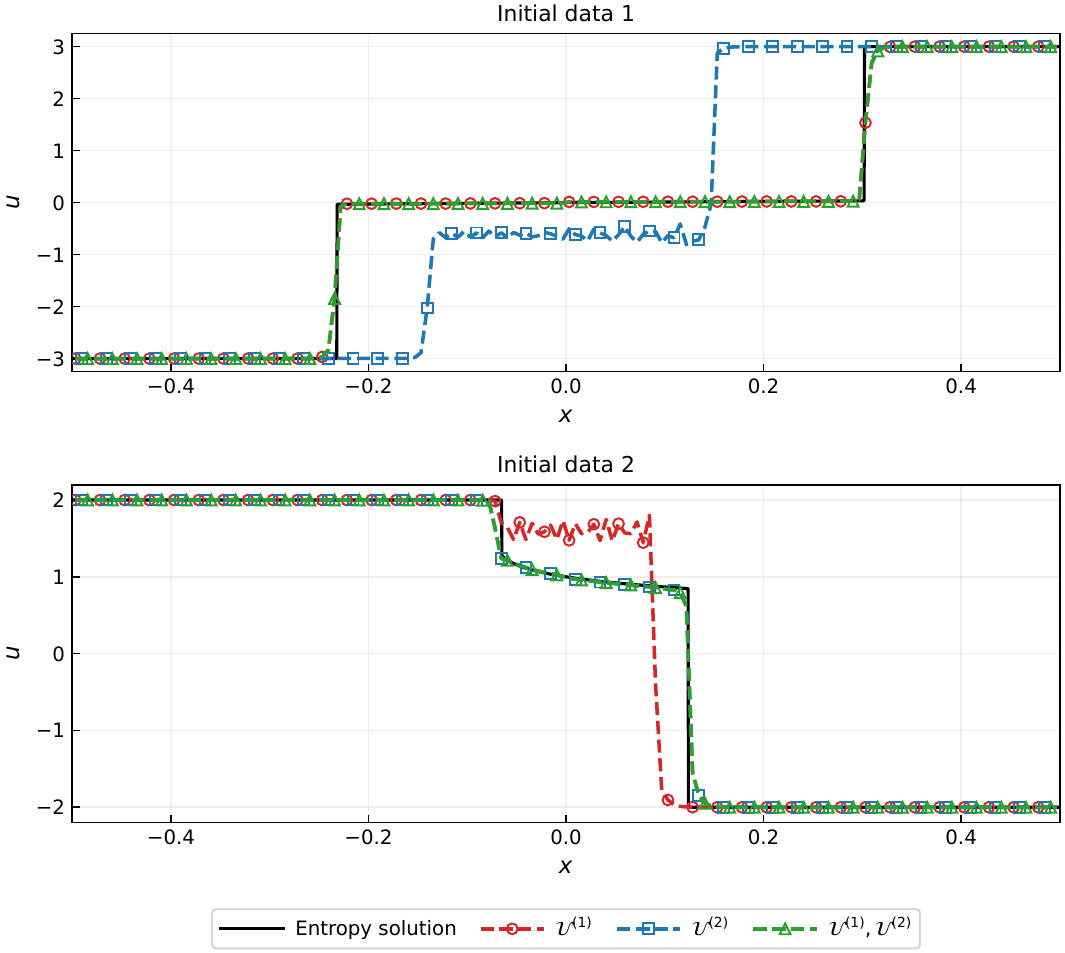}
			\caption{\Cref{ex:bl}: Buckley--Leverett equation. Solutions with different entropy constraints at $T=1$, $N=160$. }
			\label{fig:bl}
		\end{figure}
		
		\begin{table}[!htbp]
			\centering\footnotesize\setlength{\tabcolsep}{3pt}
			\caption{\Cref{ex:bl}: Buckley--Leverett equation. Errors and cell entropy residuals for both entropies, at $T=1$, $N=160$.}
			\label{tab:bl-results}
			\begin{tabular}{llr@{}lr@{}lr@{}lr@{}l}
				\toprule
				Data & Constraints & \multicolumn{2}{c}{$L^1$ error} & \multicolumn{2}{c}{$L^2$ error} & \multicolumn{2}{c}{$\max R^{(1)}$} & \multicolumn{2}{c}{$\max R^{(2)}$}\\
				\midrule
				1 & $\mathcal U^{(1)}$ & $2.021$ & ${}\times10^{-2}$ & $1.266$ & ${}\times10^{-1}$ & $3.451$ & ${}\times10^{-16}$ & $2.376$ & ${}\times10^{-2}$\\
				1 & $\mathcal U^{(2)}$ & $9.052$ & ${}\times10^{-1}$ & $1.504$ & ${}\times10^{0}$ & $2.913$ & ${}\times10^{-2}$ & $3.460$ & ${}\times10^{-16}$\\
				1 & $\mathcal U^{(1)}$ and $\mathcal U^{(2)}$ & $2.053$ & ${}\times10^{-2}$ & $1.291$ & ${}\times10^{-1}$ & $3.451$ & ${}\times10^{-16}$ & $3.460$ & ${}\times10^{-16}$\\
				2 & $\mathcal U^{(1)}$ & $1.844$ & ${}\times10^{-1}$ & $5.566$ & ${}\times10^{-1}$ & $1.716$ & ${}\times10^{-16}$ & $1.538$ & ${}\times10^{-2}$\\
				2 & $\mathcal U^{(2)}$ & $1.608$ & ${}\times10^{-2}$ & $1.193$ & ${}\times10^{-1}$ & $4.274$ & ${}\times10^{-7}$ & $1.725$ & ${}\times10^{-16}$\\
				2 & $\mathcal U^{(1)}$ and $\mathcal U^{(2)}$ & $1.605$ & ${}\times10^{-2}$ & $1.189$ & ${}\times10^{-1}$ & $1.716$ & ${}\times10^{-16}$ & $1.725$ & ${}\times10^{-16}$\\
				\bottomrule
			\end{tabular}
		\end{table}
	\end{Ex}
	
	\subsection{One-dimensional Euler equations}
	For the one-dimensional compressible Euler equations,
	\[
	\mathbf u=(\rho,\rho u,E)^{\mathrm{T}},\quad
	f(\mathbf u)=(\rho u,\rho u^2+p,u(E+p))^{\mathrm{T}},\quad
	p=(\gamma-1)\left(E-\dfrac{\rho u^2}{2}\right),
	\]
	we enforce the physical entropy pair $\UU=-\rho\log(p/\rho^\gamma)/(\gamma-1)$, $\FF=u\UU$, with $\gamma=1.4$ unless stated otherwise.
	
	\begin{Ex}[Smooth isentropic wave]\label{ex:smooth-euler}
		On $[-1,1]$ with periodic boundary conditions and $\gamma=3$, we set
		\[
		\rho(x,0)=\frac{2-\sin(\pi x)}{2\sqrt3},\quad
		u(x,0)=\frac{2+\sin(\pi x)}2,\quad
		p(x,0)=\rho(x,0)^3,
		\]
		whose exact solution~\cite{wu2025} is $\rho(x,t)=(2-w(x,t))/(2\sqrt3)$, $u(x,t)=(2+w(x,t))/2$, $p(x,t)=\rho(x,t)^3$, where $w$ solves the periodic Burgers problem $w_t+ww_x=0$ with $w(x,0)=\sin(\pi x)$. \Cref{tab:eulersmooth} shows that the ES and baseline schemes produce comparable density $L^2$ errors and convergence rates at $T=0.2$, confirming accuracy preservation for the Euler system.
		
		\begin{table}[!htbp]
			\centering
			\caption{\Cref{ex:smooth-euler}: Smooth isentropic wave. Density $L^2$ errors and convergence orders at $T=0.2$.}
			\label{tab:eulersmooth}
			\begin{tabular}{rr@{}lrr@{}lr}
				\toprule
				$N$ & \multicolumn{2}{c}{ES} & Order & \multicolumn{2}{c}{Baseline} & Order\\
				\midrule
				40 & $1.28840$ & ${}\times10^{-4}$ & -- & $1.27518$ & ${}\times10^{-4}$ & --\\
				80 & $2.34363$ & ${}\times10^{-5}$ & 2.459 & $2.34700$ & ${}\times10^{-5}$ & 2.442\\
				160 & $3.87415$ & ${}\times10^{-6}$ & 2.597 & $3.88746$ & ${}\times10^{-6}$ & 2.594\\
				320 & $6.02935$ & ${}\times10^{-7}$ & 2.684 & $6.04791$ & ${}\times10^{-7}$ & 2.684\\
				640 & $9.03971$ & ${}\times10^{-8}$ & 2.738 & $9.06131$ & ${}\times10^{-8}$ & 2.739\\
				\bottomrule
			\end{tabular}
		\end{table}
	\end{Ex}
	
	\begin{Ex}[Shu--Osher problem]\label{ex:shu-osher}
		The Shu--Osher problem~\cite{shu1989efficient} on $[-5,5]$ initializes a Mach~3 shock interacting with a sinusoidal density wave:
		\[
		(\rho,u,p)(x,0)=
		\begin{cases}
			(3.857143,2.629369,10.3333),&x<-4,\\
			(1+0.2\sin(5x),0,1),&x\ge-4.
		\end{cases}
		\]
		\cref{fig:shu-osher} compares the ES density profile at $T=1.8$ on $N=200$ cells against a 2000-cell numerical reference solution, showing sharp shock resolution and accurate post-shock density waves.
		
		\begin{figure}[!htbp]\centering
			\includegraphics[width=.90\textwidth]{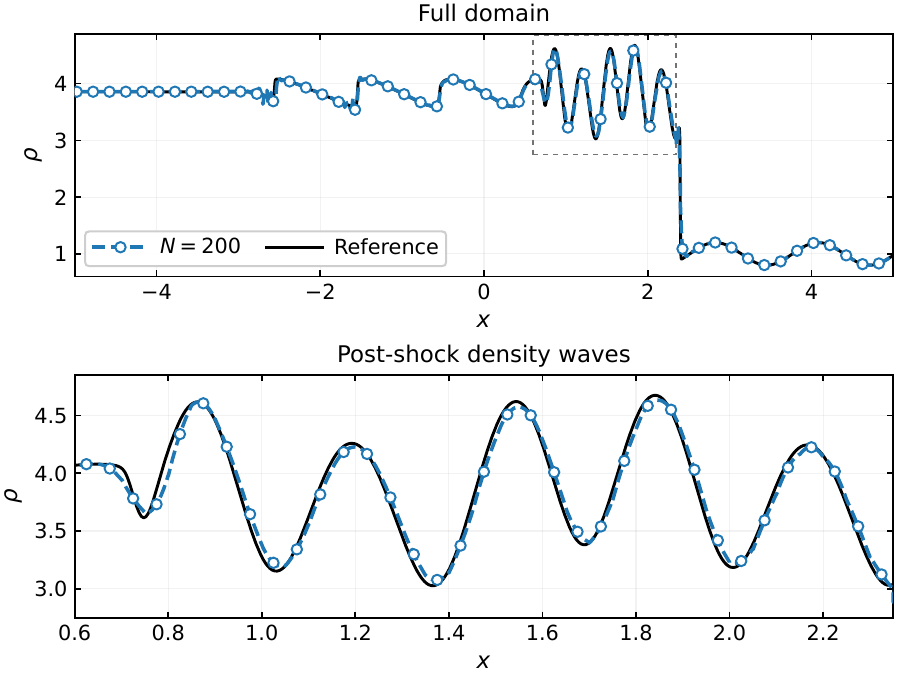}
			\caption{\Cref{ex:shu-osher}: Shu--Osher problem. Density at $T=1.8$ on $N=200$ cells. The lower panel enlarges the post-shock oscillations.}
			\label{fig:shu-osher}
		\end{figure}
	\end{Ex}

	\subsection{Two-dimensional Euler equations}
	
	\begin{Ex}[Isentropic vortex]\label{ex:vortex}
		Consider the two-dimensional isentropic vortex problem~\cite{yee1999lowdissipative} on the periodic domain $[0,10]\times[-5,5]$ with background velocity $(1,0)$ up to $T=10$. With minimum-image coordinates $X=x-5-t$ and $Y=y$, the exact solution is
		\[
		A=e^{1-X^2-Y^2},\qquad
		\vartheta=1-\dfrac{25(\gamma-1)}{16\gamma\pi^2}A^2,\qquad
		\rho=\vartheta^{1/(\gamma-1)},\quad p=\rho^\gamma,
		\]
		and $(u,v)=\bigl(1-5AY/(2\pi),5AX/(2\pi)\bigr)$. \Cref{tab:euler-vortex-comparison} shows that on $N\times N$ meshes, the ES and baseline schemes achieve comparable third-order convergence with nearly identical errors on fine grids.
		
		\begin{table}[!htbp]
			\centering
			\caption{\Cref{ex:vortex}: Isentropic vortex. Density $L^2$ errors and convergence orders at $T=10$.}
			\label{tab:euler-vortex-comparison}
			\begin{tabular}{rr@{}lrr@{}lr}
				\toprule
				$N$ & \multicolumn{2}{c}{{ES}} & Order & \multicolumn{2}{c}{{Baseline}} & {Order}\\
				\midrule
				20 & $4.05801$ & ${}\times10^{-2}$ & -- & $4.71870$ & ${}\times10^{-2}$ & --\\
				40 & $2.74242$ & ${}\times10^{-3}$ & 3.887 & $2.79639$ & ${}\times10^{-3}$ & 4.077\\
				80 & $2.98986$ & ${}\times10^{-4}$ & 3.197 & $3.00077$ & ${}\times10^{-4}$ & 3.220\\
				160 & $4.30180$ & ${}\times10^{-5}$ & 2.797 & $4.30294$ & ${}\times10^{-5}$ & 2.802\\
				\bottomrule
			\end{tabular}
		\end{table}
	\end{Ex}
	
	\begin{Ex}[Double Mach reflection]\label{ex:dmr}
		For the double Mach reflection problem~\cite{woodward1984numerical} on $[0,4]\times[0,1]$, a Mach~10 shock initially located at $x=1/6+y/\sqrt3$ separates the post-shock state $(\rho,u,v,p)=(8,8.25\sqrt3/2,-4.125,116.5)$ from the pre-shock state $(1.4,0,0,1)$. Inflow conditions are prescribed on the left boundary and on $[0,1/6]$ along the bottom boundary, a reflecting wall is imposed for $x>1/6$ along the bottom boundary, exact shock motion is prescribed at $y=1$, and outflow is used at $x=4$. \Cref{fig:dmr-density} shows the density at $T=0.2$ on $960\times240$ cells with a positivity-preserving limiter~\cite{zhang2010positivity}, capturing the Mach stems and the Kelvin--Helmholtz roll-up along the primary slip line.
		
		\begin{figure}[!htbp]\centering
			\setlength{\abovecaptionskip}{8pt}\includegraphics[width=\textwidth]{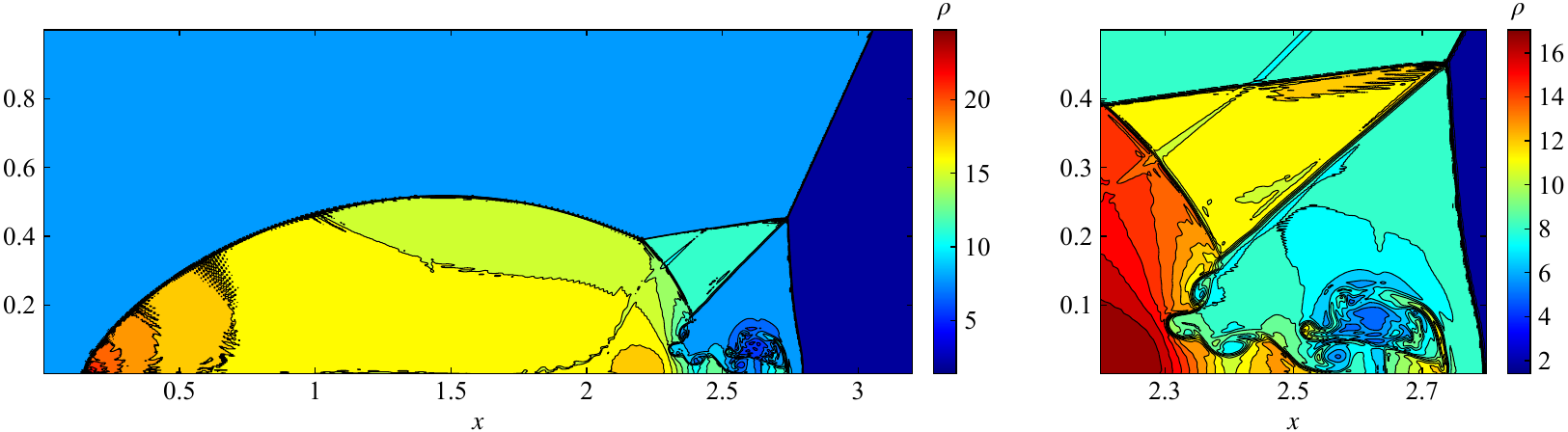}\par\nointerlineskip
			\caption{\Cref{ex:dmr}: Double Mach reflection. Density at $T=0.2$ on $960\times240$ cells. Left: overall view; right: local enlargement.}
			\label{fig:dmr-density}
		\end{figure}
	\end{Ex}

	\section{Concluding remarks}\label{sec:conclusions}
	We have developed a unified framework for high-order, fully discrete entropy-stable DG methods with compact explicit single-step time discretizations. By scaling the preliminary DG polynomial toward its cell average once at the end of each step, the method enforces a two-level strong cell entropy inequality while preserving conservation. Whereas standard SSP-RK3 DG updates can violate the necessary cell-average entropy condition at arbitrarily small CFL numbers due to stagewise stencil expansion, compact single-step methods admit nonempty intervals of feasible numerical entropy fluxes under a CFL condition by utilizing the polynomial Jensen residual to control high-order flux perturbations.
	
	Projecting a compatible high-order target entropy flux onto the feasible interval guarantees stability and preserves optimal order single-step accuracy in smooth regions, including near nondegenerate extrema. Our unified Lax--Friedrichs target construction applies across cRKDG, LWDG, and ADER-DG schemes. Moreover, the scheme accommodates multiple simultaneous entropy constraints, while ensuring that strong $L^1$ limits satisfy all enforced entropy inequalities. Numerical experiments confirm optimal convergence rates, effective multi-entropy control for nonconvex fluxes, and robustness. Future work will address extensions to unstructured meshes and more complex systems.
	
	\bibliographystyle{siamplain}
	\bibliography{references}
\end{document}